\documentclass[a4paper,reqno, 10pt]{amsart}

\usepackage{amsmath,amssymb,amsfonts,amsthm,mathrsfs}

\usepackage{verbatim,wasysym,cite}

\usepackage[latin1]{inputenc}
\usepackage{microtype}
\usepackage{color,enumitem,graphicx}
\usepackage[colorlinks=true,urlcolor=blue, citecolor=red,linkcolor=blue,
linktocpage,pdfpagelabels, bookmarksnumbered, bookmarksopen]{hyperref}
\usepackage[english]{babel}
\usepackage[symbol]{footmisc}
\renewcommand{\epsilon}{{\varepsilon}}
\usepackage{enumitem}

\numberwithin{equation}{section}
\newtheorem{theorem}{Theorem}[section]
\newtheorem{lemma}[theorem]{Lemma}
\newtheorem{remark}[theorem]{Remark}

\newtheorem{definition}[theorem]{Definition}
\newtheorem{proposition}[theorem]{Proposition}
\newtheorem{corollary}[theorem]{Corollary}
\newtheorem{claim}[theorem]{Claim}

\newcommand{\C}{\mathbb C}

\newcommand{\R}{\mathbb R}
\newcommand{\N}{\mathbb N}

\def\({\left(}
\def\){\right)}
\def\<{\left\langle}
\def\>{\right\rangle}

\def\T{\tau}
\def\Sch{{\mathcal S}}

\def\F{\mathcal F}
\def\K{\mathcal K}
\def\TT{\mathcal T}

\def\P{\mathcal P}
\def\G{\mathcal G}
\def\L{\mathcal L}
\def\EE{\mathcal E}

\def\ZZ{\mathcal Z}

\def\eps{\varepsilon}

\DeclareMathOperator{\RE}{Re}
\DeclareMathOperator{\IM}{Im}

\newcommand{\qtq}[1]{\quad\text{#1}\quad}

\begin{document}

\title[Threshold solutions]{Threshold solutions for the $3d$ cubic-quintic NLS}

\author[Alex H. Ardila]{Alex H. Ardila}
\address{Universidade Federal de Minas Gerais\\ ICEx-UFMG\\ CEP
 30123-970\\ MG, Brazil} 
\email{ardila@impa.br}

\author{Jason Murphy}
\address{Missouri University of Science \& Technology \\ Rolla, MO, USA}
\email{jason.murphy@mst.edu}

\begin{abstract} We study the cubic-quintic NLS in three space dimensions.  It is known that scattering holds for solutions with mass-energy in a region corresponding to positive virial, the boundary of which is delineated both by ground state solitons and by certain rescalings thereof.  We classify the possible behaviors of solutions on the part of the boundary attained solely by solitons.  In particular, we show that non-soliton solutions either scatter in both time directions or coincide (modulo symmetries) with a special solution, which scatters in one time direction and converges exponentially to the soliton in the other. 
\end{abstract}


\maketitle

\medskip

\section{Introduction}
\label{sec:intro}
We study the cubic-quintic nonlinear Schr\"odinger equation (NLS) 
\begin{equation}\tag{NLS}\label{NLS}
\begin{cases} 
(i\partial_{t}+\Delta)u=-|u|^{2}u+|u|^{4}u,\\
u(0)=u_{0}\in H^{1}(\R^{3}),
\end{cases} 
\end{equation}
where $u: \R\times\R^{3}\rightarrow \C$.  This is the Hamiltonian evolution corresponding to the energy 
\[
E(u)=\int_{\R^{3}}\tfrac{1}{2}|\nabla u|^{2}
-\tfrac{1}{4}|u|^{4}+\tfrac{1}{6}|u|^{6}\,dx.
\] 
Solutions to \eqref{NLS} additionally preserve the mass, defined by
\[
M(u)=\int_{\R^{3}}|u|^{2}\,dx.
\]
Using the conservation of mass and energy, one can show that \eqref{NLS} is globally well-posed in $H^1$, with solutions remaining uniformly bounded in $H^1$ over time (see \cite{Zhang2006}).  Our interest is in the long-time behavior of solutions to \eqref{NLS}, with a focus on the question of scattering, i.e. whether an initial condition $u_0\in H^1(\R^3)$ leads to a solution $u$ obeying
\begin{equation}\label{scattering}
\lim_{t\to\pm\infty}\|u(t)-e^{it\Delta}u_\pm\|_{H^1} = 0 \qtq{for some}u_\pm\in H^1.
\end{equation}

This problem was studied in \cite{KillipOhPoVi2017}, which identified a scattering region in the mass-energy plane corresponding to the region of positive virial.  The boundary of this region is realized in part by (non-scattering) soliton solutions, but also in part by certain rescalings of solitons (which are not themselves soliton solutions to \eqref{NLS}). The behavior of solutions in the part of the boundary \emph{not} realized by any soliton was further studied in \cite{KillipMurphyVisan2020}, which established scattering in an open neighborhood of this part of the boundary (and, in particular, across the `virial threshold').  Our interest in this paper is to study the behavior of solutions with mass-energy belonging to the part of the boundary realized exclusively by solitons.  We classify all possible behaviors: other than the soliton (corresponding to zero virial), solutions either scatter as $t\to\pm\infty$ or coincide (modulo symmetries) with a special solution, which scatters in one direction and converges exponentially to the soliton in the other.  All three possible behaviors do occur.  Our result is closely related to the threshold classifications appearing in works such as \cite{DuyMerle2009, DuyckaertsRou2010, LiZhang1, CFR, CM, YZZ}; essentially, we obtain the usual classification, but with all blowup behavior removed.

To state our main result precisely, we first need to introduce the variational problem from \cite{KillipOhPoVi2017} that is used to define the scattering region. Writing $V$ for the virial functional,
\begin{equation}\label{Virial-Functional}
V(f)=\| f\|^{2}_{\dot{H}_{x}^{1}}+\|f\|^{6}_{L_{x}^{6}}
-\tfrac{3}{4}\|f\|^{4}_{L_{x}^{4}},
\end{equation}
we set
\begin{equation*}
\EE(m):=\inf\left\{E(f): f\in H^{1}(\R^{3}),\, M(f)=m \,\, \mbox{and}\,\, V(f)=0 \right\}.
\end{equation*}
The open set $\K\subset \R^{2}$ is then defined by
\begin{equation}
\label{Virial}
\K:= \left\{(m,e): 0<m<m_2 \,\, \mbox{and} \,\, 0<e<\EE(m)\right\},
\end{equation}
where is $m_2=M(Q_{1})$, with $Q_1$ an optimizer of a certain Gagliardo--Nirenberg--H\"older inequality (cf. Theorem~\ref{TheInverGN} below with $\alpha=1$).  The scattering result of \cite{KillipOhPoVi2017} can then be stated as follows:

\begin{theorem}[Scattering with positive virial, \cite{KillipOhPoVi2017}]\label{MainTheorem} If $(M(u_{0}),E(u_{0}))\in \K$, then the solution to \eqref{NLS} is global and scatters as $t\to\pm\infty$. 
\end{theorem}

The region $\mathcal{K}$ is represented in the following figure, imported from \cite{KillipMurphyVisan2020}.

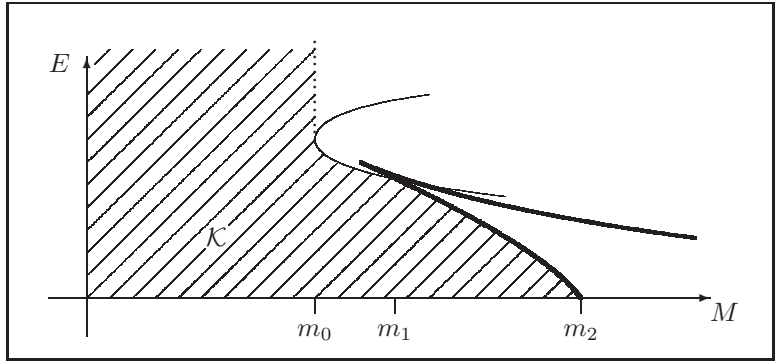
\begin{figure}[h]
\noindent
\begin{center}
\fbox{
\setlength{\unitlength}{1mm}
\begin{picture}(95,45)(-7,-7)
\put(0,-5){\vector(0,1){37}}\put(-5,30){$E$} 
\put(-5,0){\vector(1,0){87}}\put(82,-3){$M$} 
\put(30,0){\line(0,-1){2}}\put(30,-5){\hbox to 0mm{\hss$m_0$\hss}}
\put(40.5,0){\line(0,-1){2}}\put(40.5,-5){\hbox to 0mm{\hss$m_1$\hss}}
\put(65,0){\line(0,-1){2}}\put(65,-5){\hbox to 0mm{\hss$m_2$\hss}}
%
%
\linethickness{0.4mm}
\qbezier(36,18)(60,7)(65,0)
\qbezier(36,18)(50,12)(80,8)
%
%
\linethickness{0.05mm}
\qbezier(30,21)(30,25)(45,27)
\qbezier(30,21)(30,16.5)(55,13.5)
%
%
\multiput(30,21)(0,1){14}{\circle*{0.1}}
%
%
\qbezier(1,31)(02,32)(03,33)
\qbezier(0,27)(01,28)(06,33)\qbezier(0,24)(01,25)(09,33)
\qbezier(0,21)(03,24)(12,33)\qbezier(0,18)(04,22)(15,33)
\qbezier(0,15)(06,21)(18,33)\qbezier(0,12)(07,19)(21,33)
\qbezier(0,09)(09,18)(24,33)\qbezier(0,06)(10,16)(27,33)
\qbezier(0,03)(12,15)(30,33)
\qbezier(0,00)(15,15)(30,30)\qbezier(03,0)(16,13)(30,27)\qbezier(06,0)(18,12)(30,24)
\qbezier(09,0)(13,4)(15.6,6.6)\qbezier(18.7,9.7)(25,16)(30,21)\put(15.7,7){$\mathcal K$}
%
%
\qbezier(12,0)(15, 3)(31.1,19.1)
\qbezier(15,0)(18, 3)(33.0,18.0)
\qbezier(18,0)(21, 3)(35.2,17.2)
\qbezier(21,0)(24, 3)(37.5,16.5)
\qbezier(24,0)(27, 3)(39.9,15.9)
%
%
\qbezier(27,0)(30, 3)(42.1,15.1)
\qbezier(30,0)(33, 3)(44.1,14.1)
\qbezier(33,0)(36, 3)(46.1,13.1)
\qbezier(36,0)(39, 3)(48.1,12.1)
\qbezier(39,0)(42, 3)(50.0,11.0)
\qbezier(42,0)(45, 3)(52.0,10.0)
\qbezier(45,0)(48, 3)(53.9, 8.9)
\qbezier(48,0)(51, 3)(55.7, 7.7)
\qbezier(51,0)(54, 3)(57.5, 6.5)
\qbezier(54,0)(57, 3)(59.3, 5.3)
\qbezier(57,0)(60, 3)(61.0, 4.0)
\qbezier(60,0)(63, 3)(62.6, 2.6)
\qbezier(63,0)(64, 1)(64.1, 1.1)
\end{picture}
}
\end{center}
\caption{Figure 1 from \cite{KillipMurphyVisan2020}. Soliton solutions correspond to the heavy curve; non-soliton virial obstructions correspond to the light curve.}
\label{F:1}
\end{figure}

It was shown in \cite{KillipOhPoVi2017} that the boundary of $\mathcal{K}$ (at mass $m\in(m_0,m_2)$) is attained either at a soliton or a rescaled soliton.  Here the (ground state) solitons refer to the unique, non-negative, radially symmetric solutions $P_\omega\in H^1$ to
\begin{equation}\label{GroundS}
-\omega P_\omega +\Delta P_\omega =-P_\omega^3+ P_\omega^5,\quad \omega\in(0,\tfrac3{16}),
\end{equation}
while the rescaled solitons $R_\omega$ are given by 
\[
R_\omega(x):=\sqrt{\tfrac{1+\beta(\omega)}{4\beta(\omega)}}\,P_\omega\!\Bigl(\tfrac{3(1+\beta(\omega))}{4\sqrt{3\beta(\omega)}}x\Bigr),
\qtq{with} \beta(\omega):=\tfrac{\|P_\omega\|_{L^6}^6}{\|\nabla P_\omega\|_{L^2}^2}.
\]
It was further shown in \cite{KillipOhPoVi2017} that the left-most part of $\partial\K$ is attained only by rescaled solitons, while the the right-most part is attained only by solitons.  In fact, it is conjectured (and supported by numerics) that there exists a mass $m_1$ dividing the boundary into the rescaled soliton part and soliton part.  It is also natural to conjecture that for a given mass, the corresponding soliton or rescaled soliton delineating the boundary is unique; we will work under this assumption below. 
 
The behavior of solutions around the left-most part of $\partial\mathcal{K}$ was studied in \cite{KillipMurphyVisan2020}, where the following result was obtained:

\begin{theorem}[Crossing the virial threshold, \cite{KillipMurphyVisan2020}]\label{FreeNLS}  There exists an open set $\mathcal{B}\subset\R^2$ containing $\mathcal{K}$ such that all solutions with mass-energy in $\mathcal{B}$ scatter as $t\to\pm\infty$.  The set $\mathcal{B}$ satisfies (i) there exists $m^*>m_0$ so that $(0,m^*)\times(0,\infty)\subset\mathcal{B}$, and (ii) any $(m,e)\in\partial\K$ not achieved by a soliton belongs to $\mathcal{B}$. 
\end{theorem}

Our interest in this paper is the behavior of solutions on the right part of $\partial\mathcal{K}$.

\begin{definition} We let $\partial\mathcal{K}_s$ denote the set of $(m,e)\in\partial\K$ with $e>0$ such that:
\begin{itemize}
\item[(i)] $(m,e)=(M(P_\omega),E(P_\omega))$ for some unique $\omega\in(0,\tfrac3{16})$; 
\item[(ii)] $(m,e)\neq(M(R_\omega),E(R_\omega))$ for any $\omega\in(0,\tfrac{3}{16})$.
\end{itemize} 
\end{definition}

As described above, the numerics from \cite{KillipMurphyVisan2020} suggest that $\partial\mathcal{K}_s$ coincides with the portion of $\partial\K$ with $m\in(m_1,m_2)$. The existence in (i) and nonexistence in (ii) are guaranteed for all $m$ sufficiently close to $m_2$ by the results of \cite{KillipOhPoVi2017}; however, the uniqueness in (i) must be taken as an assumption throughout the paper.

Our first result is the existence of a solution scattering in one time direction and converging exponentially to $P_{\omega}$ in the other. 

\begin{theorem}\label{TH1}
If $(m,e)=(M(P_\omega),E(P_\omega))\in \partial\K_{s}$, then there exists a global solution  $\G_{\omega}$ of \eqref{NLS}  with $(M(\G_{\omega}), E(\G_{\omega}))=(m,e)$ satisfying the following:
\begin{itemize}
\item $V(\G_{\omega}(t))>0$ for all $t\in \R$;
\item $\G_{\omega}$ scatters as $t\to-\infty$;
\item there exist constants $C>0$ and $c>0$ such that
\[
\|\G_{\omega}(t)-e^{i \omega t}P_{\omega}\|_{H^{1}_{x}}\leq C e^{-c t}\quad \text{for all $t\geq 0$.}
\]
\end{itemize}
\end{theorem}

Using the solution $\G_\omega$ in Theorem~\ref{TH1}, we can classify all possible behaviors for solutions with mass-energy in $\partial\K_s$. 
\begin{theorem}[Threshold behaviors on $\in\partial\K_s$]\label{ScatSoliton}  Suppose $u$ is a solution to \eqref{NLS} with $(M(u),E(u))=(M(P_\omega),E(P_\omega))\in \partial\K_s$. 
\begin{enumerate}[label=\rm{(\roman*)}]
\item If  $V(u|_{t=0})=0$, then $u(t)=e^{i \omega t}P_{\omega}$ up to the symmetries of the equation.
\item If $V(u|_{t=0})>0$, then either $u$  scatters in both time directions or  $u=\G_{\omega}$ up to the symmetries of the equation.
\end{enumerate}
\end{theorem}

Theorems~\ref{TH1}~and~\ref{ScatSoliton} are similar to the threshold classification results appearing in works on pure power NLS (see e.g. \cite{DuyckaertsRou2010, CFR}); essentially, we obtain the standard classification with all blowup behavior removed. As mentioned above, all of the behaviors described in Theorem~\ref{ScatSoliton} actually do occur; in particular, we demonstrate the existence of scattering solutions in Proposition~\ref{P:exist-scatter}.  Our results do not apply to the case of zero energy solutions; indeed, one can show that scattering is not possible in this case (see Remark~\ref{RB2}).

In the rest of the introduction, let us briefly describe the organization of the paper and the strategy of proof for Theorems~\ref{TH1}~and~Theorem~\ref{ScatSoliton}.  

In Section~\ref{S:Spectral}, we analyze the spectrum of the operator arising from the linearization of \eqref{NLS} around the ground state.  In Section~\ref{S:Modula}, we then carry out the modulation analysis around the ground state.  In Section~\ref{S: Concentrated}, we use prove that solutions with $(M(u),E(u))\in\partial\mathcal{K}_s$ that do not scatter forward in time must converge exponentially to a ground state as $t\to\infty$ modulo symmetries.  This involves some standard concentration-compactness techniques to establish compactness of the orbit, combined with the modulation analysis and the (modulated) virial estimate. 

Section~\ref{S: stableS} contains the existence and uniqueness of solutions converging exponentially to the ground state (which ultimately yields the solution $\mathcal{G}_\omega$ appearing in Theorem~\ref{TH1}).  Using the spectral properties of the linearized operator, we first construct suitable approximate solutions, which we then upgrade to true solutions via a fixed point argument.  Section~\ref{S:Rigid} is then devoted to the proof of a rigidity-type result, namely, that solutions converging exponentially the ground state must coincide with one of the solutions constructed in Section~\ref{S: stableS} (or the ground state itself).  We further show that (modulo time translation and rotation), all of the  solutions constructed in Section~\ref{S: stableS} (other than the ground state) are in fact that same.

With all of the preceding results in hand, we can quickly complete the proofs of the main results in Section~\ref{S:proof}.  The appendix then contains the proof of a technical lemma from Section~\ref{S:Spectral}, as well as a demonstration of the existence of scattering solutions on $\partial\mathcal{K}_s$.

\subsection*{Notation}
We write $A\lesssim B$ or $B\gtrsim A$  to denote $A\leq CB$ for some positive constant $C$. We also write $A\sim B$ when $A \lesssim B \lesssim A$.
For a function $u:I\times \R^{3}\rightarrow \C$, $I\subset \R$, we write
\[ \|  u \|_{L_{t}^{q}L^{r}_{x}(I\times \R^{3})}=\|  \|u(t) \|_{L^{r}_{x}(\R^{3})}  \|_{L^{q}_{t}(I)}
 \]
with $1\leq q\leq r\leq\infty$. 
We say that the pair $(q,r)$ is admissible when $2\leq q,r\leq\infty$ and $\frac{2}{q}+\frac{3}{r}=\frac{3}{2}$. 
We write $\<\nabla\>=(1-\Delta)^{1/2}$ and we define the Sobolev norms 
\[
\|u\|_{H^{s,r}(\R^{3})}:=\|\<\nabla\>^{s} u\|_{L^{r}_{x}(\R^{3})}.
\]
If $\alpha(t)$ and $\beta(t)$ are two positive functions of $t$, we write $\alpha=O(\beta)$ when $\alpha\leq C\beta$ for some constant $C>0$,
and $\alpha \sim \beta$ when $\alpha=O(\beta)$ and $\beta=O(\alpha)$.

\section{Preliminaries}\label{S1:preli} 
In this section, we record some preliminary results that will be used throughout the paper. 

\subsection{Well-posedness and stability}
We first recall the Strichartz estimates for the linear Schr\"odinger equation.  We call a pair $(q,r)$ admissible if $2\leq q,r\leq\infty$ and $\tfrac{2}{q}+\tfrac{3}{r}=\tfrac{3}{2}$. 
\begin{lemma}[Strichartz estimates, \cite{GV, KT, Strichartz}]  Let $(q,r)$ and $(\tilde q,\tilde r)$ be admissible. Then the solution $u$ to the equation $(i\partial_{t}+\Delta) u=F$ with $u|_{t_0}=u_0$ satisfies
\begin{equation}\label{Strichartz}
\| u  \|_{L_{t}^{q}L^{r}_{x}(I\times \R^{3})} \lesssim \| u_{0}\|_{L^{2}_{x}(\R^{3})}+ \| F\|_{L_{t}^{\tilde{q}'}L^{\tilde{r}'}_{x}(I\times \R^{3})},
\end{equation}
for any interval $t_0\in I\subset \R$. 
\end{lemma}

We also have the following global well-posedness result from \cite{Zhang2006}. 
\begin{theorem}[Global well-posedness, \cite{Zhang2006}]\label{Th1} For any $u_0\in H^1(\R^3)$, the solution to \eqref{NLS} is global and obeys $u\in C(\R;H^1(\R^3))$.  Moreover, the mass, energy, and momentum are conserved, where the momentum is defined by
\[
\ZZ(u(t)):=\int_{\R^{3}}2\IM(\overline{u(t,x)}\nabla u(t,x))\,dx.
\]
\end{theorem}

For an interval $I\subset\R$, we define
\[
\|u\|_{S^{0}(I)}:=\sup\{\|u\|_{L^{q}_{t}L^{r}_{x}(I\times\R^{3})}: \quad \text{$(q,r)$ is admissible}\}.
\]
We then have the following stability result (cf. \cite{KillipOhPoVi2017}): 

\begin{lemma}[Stability]\label{S: result}
Let $I\subset \R$ be a time interval containing $t_{0}$. Suppose that  $\tilde{u}:I\times\R^{3} \to \C$ solves
\[ (i\partial_{t}+\Delta) \tilde{u}=|\tilde{u}|^{4}\tilde{u}-|\tilde{u}|^{2}\tilde{u}+e, \quad 
\tilde{u}(t_{0})=\tilde{u}_{0}\]
for some function $e:I\times\R^{3}\rightarrow \C$. Let $u_{0}\in H^{1}(\R^{3})$ and suppose
\begin{align*}
	\| \tilde{u}  \|_{L_{t}^{\infty}H^{1}_{x}(I\times\R^{3})}&\leq A\\
	\| \tilde{u}  \|_{L_{t,x}^{10}(I\times\R^{3})}&\leq L
\end{align*}
for some $A$, $L>0$.   
There exists $\epsilon_{0}=\epsilon_{0}(\mbox{A,L})>0$ such that if $0<\epsilon<\epsilon_{0}$ and
\begin{align*}
	\|u_{0}-\tilde{u}_{0} \|_{\dot{H}^{1}_{x}}&\leq \epsilon\\
	\|\nabla e  \|_{L_{t,x}^{\frac{10}{7}}(I\times\R^{3})}&\leq \epsilon,
\end{align*}
then there exists a unique solution $u$ to \eqref{NLS} with $u(t_{0})=u_{0}$ 
satisfying
\begin{equation*}
\|\nabla (u-\tilde{u})\|_{S^{0}(I)}\leq C(A,L)\epsilon.
\end{equation*}
\end{lemma}

\subsection{Variational analysis}  We record here some of the variational analysis that played a central role in \cite{KillipOhPoVi2017}.  As in that work, we utilize the notation 
\[
\beta(\omega):=\frac{\|P_{\omega}\|^{6}_{L^{6}}}{\|\nabla P_{\omega}\|^{2}_{L^{2}}}. 
\]

We first have the following: 
\begin{theorem}[$\alpha$-Gagliardo-Nirenberg-H\"older inequality, \cite{KillipOhPoVi2017}]\label{TheInverGN}
Let $\alpha\in(0,\infty)$ and
\begin{equation}\label{GNI}
C^{-1}_{\alpha}:=\inf_{ f\in {H}^{1}\setminus\left\{0\right\}}
\frac{\|f\|_{L^{2}}\| f\|^{\frac{3}{1+\alpha}}_{\dot{H}^{1}}\|f\|^{\frac{3\alpha}{1+\alpha}}_{L^{6}}}
{\|f\|^{4}_{L^{4}}}.
\end{equation}
Then $C_{\alpha}\in (0,\infty)$ and the infimum \eqref{GNI} is attained by a function of the form
$f(x)=\lambda P_{\omega}(r(x-x_{0}))$, where $ P_{\omega}$  is a radial positive solution to the
stationary problem \eqref{GroundS} with $\beta(P_{\omega})=\alpha$, $\lambda\in \R$, $r>0$ and $x_{0}\in \R^{3}$.
\end{theorem}

\begin{remark} The optimal constant $C_{1}$ in \eqref{GNI} is given by
\[
C_{1}=\tfrac{8}{3}\tfrac{1}{\|Q_{1}\|_{L^{2}}},
\]
where $Q_{1}$ denotes a ground state solution optimizing \eqref{GNI}. In particular, we have
\[
\|f\|^{4}_{L^{4}}\leq \tfrac{8}{3}\tfrac{\|f\|_{L^{2}}}{\|Q_{1}\|_{L^{2}}}\|f\|^{3/2}_{L^{6}}\|\nabla f\|^{3/2}_{L^{2}}.
\]
Moreover, using Young's inequality we get
\begin{equation}\label{BoundH}
\begin{aligned}
E(f)&\geq \tfrac{1}{2}\|f\|^{2}_{\dot{H}^{1}}+\tfrac{1}{6}\|f\|^{6}_{L^{6}}-\tfrac{2}{3}
\tfrac{\|f\|_{L^{2}}}{\|Q_{1}\|_{L^{2}}}\|f\|^{\frac{3}{2}}_{\dot{H}^{1}}\|f\|^{\frac{3}{2}}_{L^{6}}\\
&\geq \big(1-\tfrac{\|f\|_{L^{2}}}{\|Q_{1}\|_{L^{2}}} \big)\left[\tfrac{1}{2}\|f\|^{2}_{\dot{H}^{1}}+\tfrac{1}{6}\|f\|^{6}_{L^{6}}\right].
\end{aligned}
\end{equation}
\end{remark}

The ground states $P_{\omega}$ satisfy the following identities:
\begin{equation}\label{PhIden}
\begin{split}
&\|P_{\omega}\|^{2}_{L^{2}}=\tfrac{\beta({\omega})+1}{3\omega}\|\nabla P_{\omega}\|^{2}_{L^{2}},\quad
\|P_{\omega}\|^{4}_{L^{4}}=\tfrac{4[\beta({\omega})+1]}{3}\|\nabla P_{\omega}\|^{2}_{L^{2}},\\
&\|P_{\omega}\|^{6}_{L^{6}}=\beta(\omega)\|\nabla P_{\omega}\|^{2}_{L^{2}}
\quad \text{and}\quad
E(P_{\omega})=\tfrac{[1-\beta({\omega})]}{6}\|\nabla P_{\omega}\|^{2}_{L^{2}}.
\end{split}
\end{equation}

\begin{lemma}\label{PositiveV}
Let $u_{0}\in H^{1}(\R^{3})$ satisfy $(M(u_{0}), E(u_{0}))\in \partial\K_{s}$ and $V(u_{0})>0$.
Then the corresponding solution $u(t)$ of \eqref{NLS} is uniformly bounded in $H^{1}(\R^{3})$
and $V(u(t))>0$ for all $t\in \R$.
\end{lemma}
\begin{proof}
First, notice that
\[
\int_{\R^{3}}\tfrac{1}{2}|\nabla u|^{2}+\tfrac{1}{6}|u|^{2}\big(|u|^{2}-\tfrac{3}{4}\big)^{2}dx=E(u)+\tfrac{3}{32}M(u).
\]
This implies that $\|u(t)\|^{2}_{L^{2}}+\|u(t)\|^{2}_{\dot{H}^{1}} \lesssim M(u_{0})+E(u_{0})$ for all $t\in \R$.  

On the other hand, suppose by contradiction that $V(u(t_{0}))=0$ for some $t_{0}\in \R$. Since 
$(M(u(t_{0})), E(u(t_{0})))\in \partial\K_{s}$, by the variational characterization given in \cite[Theorem 5.6]{KillipOhPoVi2017}
we infer that $u(t_{0},x)=e^{i\theta}P_{\omega}(x+x_{0})$ for some $\omega\in (0, \tfrac{3}{16})$, $\theta\in \R$ and $x_{0}\in \R^{3}$. 
By uniqueness of the solution, we have $V(u_{0})=0$, which is a contradiction. \end{proof}

\begin{remark}\label{ReEMG}
Let $f\in H^{1}(\R^{3})$. We observe that if
\begin{align}\label{EMG}
	E(f)=E(P_{\omega}), \quad 	M(f)=M(P_{\omega}) \quad \text{and}
	\quad \|\nabla f\|^{2}_{L^{2}}=\|\nabla P_{\omega}\|^{2}_{L^{2}},
\end{align}
then $f(x)=e^{i \theta}P_{\omega}(x-x_{0})$ for some $\theta\in \R$ and $x_{0}\in \R^{3}$. Indeed, 
by \eqref{EMG} we deduce
\[
\int_{\R^{3}}\tfrac{\omega}{2}|f|^{2}+
\tfrac{1}{6}|f|^{6}-\tfrac{1}{4}|f|^{4}dx=
\int_{\R^{3}}\tfrac{\omega}{2}|P_{\omega}|^{2}+
\tfrac{1}{6}|P_{\omega}|^{6}-\tfrac{1}{4}|P_{\omega}|^{4}dx.
\]
Since $\|\nabla f\|^{2}_{L^{2}}=\|\nabla P_{\omega}\|^{2}_{L^{2}}$, the variational characterization of $P_{\omega}$ given in \cite[(2.8)]{KillipOhPoVi2017} and
Theorem 2.2 in \cite{KillipOhPoVi2017} (uniqueness of solitons) guarantees that $f$ must agree with $P_{\omega}$
up to translation and a phase rotation. 

In particular, we see that if
\[
E(u_{0})=E(P_{\omega}), \quad 	M(u_{0})=M(P_{\omega}) \quad \text{and} \quad V(u_{0})>0,
\]
then either
\[
\|\nabla u(t)\|^{2}_{L^{2}}<\|\nabla P_{\omega}\|^{2}_{L^{2}}\quad \text{$\forall t\in \R$ or}
\quad
\|\nabla u(t)\|^{2}_{L^{2}}>\|\nabla P_{\omega}\|^{2}_{L^{2}}\quad \text{$\forall t\in \R$},
\]
where $u(t)$ is the corresponding solution to \eqref{NLS} with initial data $u_{0}$.
\end{remark}


\begin{lemma}\label{PriModula}
Suppose $(m,e)\in \partial\K_{s}$.  If the sequence $\left\{f_{n}\right\}\subset H^{1}(\R^{3})$ obeys
$M(f_{n})=m$, $E(f_{n})=e$ and $V(f_{n})\to 0$ as $n\to \infty$, then there exists $\theta_{n}\in \R$
and $x_{n}\in \R^{3}$ such that
\[
\|f_{n}-e^{i \theta_{n}}P_{\omega}(\cdot-x_{n})\|_{H^{1}}\to 0
\]
as $n\to \infty$, where $\omega\in (0, \tfrac{3}{16})$ is unique.
\end{lemma}
\begin{proof}
Following the same argument developed in the proof of \cite[Lemma 3.2]{KillipMurphyVisan2020}, by Theorem 5.6  and $(m,e)\in\partial\K_s$ we obtain that 
there exist a unique $\omega\in (0, \frac{3}{16})$,  $\theta_{n}$ and $x_{n}$ such that
\[
e^{-i \theta_{n}}f_{n}(x+x_{n})\to P_{\omega}(x) \quad \text{in} \quad H^{1}(\R^{3}).
\]
\end{proof}

\subsection{Linear profile decomposition}
The following appears as \cite[Theorem~7.5]{KillipOhPoVi2017}.

\begin{proposition}\label{ProfL} Let $\left\{f_{n}\right\}$ be a bounded sequence in $H^{1}(\R^{3})$. The following holds up to a subsequence:

There exist $J^{\ast}\in \left\{0,1,2,\ldots\right\}\cup\left\{\infty\right\}$, non-zero profiles 
$\{ \psi^{j}\}^{J^{\ast}}_{j=1}\subset \dot{H}^{1}(\R^{3})$ and parameters
\[
\left\{\lambda^{j}_{n}\right\}_{n\in \N}\subset (0,1],\quad 
\left\{t^{j}_{n}\right\}_{n\in \N}\subset \R
\quad \text{and}\quad 
\left\{x^{j}_{n}\right\}_{n\in \N}\subset \R^{3}
\]
so that for each finite $1\leq J\leq J^{\ast}$, we can write
\begin{equation}\label{Dcom}
f_{n}=\sum^{J}_{j=1} \psi_{n}^{j}+W^{J}_{n},
\end{equation}
with
\begin{equation}\label{fucti}
 \psi_{n}^{j}(x):=
\begin{cases} 
[e^{it^{j}_{n}\Delta} \psi^{j}](x-x^{j}_{n}), &\mbox{if $\lambda^{j}_{n}\equiv 1$},\\
(\lambda^{j}_{n})^{-\frac{1}{2}}[e^{it^{j}_{n}\Delta} P_{\geq (\lambda^{j}_{n})^{\theta}} \psi^{j}]\( \frac{x-x^{j}_{n}}{\lambda^{j}_{n}}\),
&\mbox{if $\lambda^{j}_{n}\rightarrow 0$},
\end{cases}
\end{equation}
for some $0<\theta<1$, satisfying the following statements:
\begin{itemize}
	\item $\lambda^{j}_{n}\equiv 1$ or $\lambda^{j}_{n}\rightarrow 0$ and $t^{j}_{n}\equiv 0$ or $t^{j}_{n}\rightarrow\pm\infty$,
	\item if $\lambda^{j}_{n}\equiv 1$ then $\left\{ \psi^{j}\right\}^{J^{\ast}}_{j=1}\subset L_{x}^{2}(\R^{3})$
\end{itemize}
for each $j$.  Moreover, we have:
\begin{itemize}[leftmargin=5mm]
	\item Smallness of the reminder: 
	\begin{equation}\label{Reminder}
\lim_{J\to J^{\ast}}\limsup_{n\rightarrow\infty}\|e^{it\Delta}W^{J}_{n}\|_{L^{10}_{t,x}(\R\times\R^{3})}=0.
   \end{equation}
		\item Weak convergence property:
		\begin{equation}\label{WeakConver}
e^{-it^{j}_{n}\Delta}[(\lambda^{j}_{n})^{\frac{1}{2}}W^{J}_{n}(\lambda^{j}_{n}x+x^{j}_{n})]\rightharpoonup 0\quad \mbox{in}\,\,
\dot{H}^{1}_{x}, \quad \mbox{for all $1\leq j\leq J$.}
     \end{equation}
			\item Asymptotic Pythagorean expansions:
		\begin{align}\label{MassEx}
		&\sup_{J}\lim_{n\to\infty}\big[M(f_{n})-\sum^{J}_{j=1}M( \psi^{j}_{n})-M(W^{J}_{n})\big]=0,\\\label{EnergyEx}
		&\sup_{J}\lim_{n\to\infty}\big[E(f_{n})-\sum^{J}_{j=1}E( \psi^{j}_{n})-E_{a}(W^{J}_{n})\big]=0.
		\end{align}
	\item Asymptotic orthogonality: for all $1\leq j\neq k\leq J^{\ast}$
	\begin{equation}\label{Ortho}
\lim_{n\rightarrow \infty}\left[ \frac{\lambda^{j}_{n}}{\lambda^{k}_{n}}+\frac{\lambda^{k}_{n}}{\lambda^{j}_{n}} 
+\frac{|x^{j}_{n}-x^{k}_{n}|^{2}}{\lambda^{j}_{n}\lambda^{k}_{n}}+
\frac{|t^{j}_{n}(\lambda^{j}_{n})^{2}-t^{k}_{n}(\lambda^{k}_{n})^{2}|}{\lambda^{j}_{n}\lambda^{k}_{n}}\right]=\infty.
\end{equation}		
	\end{itemize}
\end{proposition}

\subsection{Virial identity}
Given $R\geq 1$, we define
\[
w_{R}(x)=R^{2}\phi(\tfrac{x}{R})
\quad \text{and}\quad w_{\infty}(x)=|x|^{2},
\]
where $\phi$ is a real-valued and radial function satisfying
\[
\phi(x)=
\begin{cases}
|x|^{2},& \quad |x|\leq 1\\
0,& \quad |x|\geq 2,
\end{cases}
\quad \text{with}\quad 
|\partial_{x}^{\alpha}\phi(x)|\lesssim |x|^{2-|\alpha|}
\]
and $\partial_{r}\phi \geq 0$. Here, $\partial_{r}$ denotes the radial derivative.

We  also define the functional
\[
\P_{R}[u]=2\IM\int_{\R^{3}} \overline{u} \nabla u \cdot \nabla w_{R}  dx.
\]

\begin{lemma}\label{VirialIden}
Let $R\in [1, \infty]$. Let $u(t)$  be a solution of \eqref{NLS}. Then we have
\begin{equation}\label{LocalVirial}
\frac{d}{d t}\P_{R}[u]=F_{R}[u(t)],
\end{equation}
where
\begin{align*}
F_{R}[u]= \int_{\R^{3}} 4\RE \overline{u_{j}}u_{k}\partial_{jk}[w_{R}]-
|u|^{4}\Delta w_{R}+\tfrac{4}{3}|u|^{6}\Delta w_{R}-\Delta\Delta w_{R} |u|^{2}dx.
\end{align*}
Note that if $R=\infty$, then we have $F_{\infty}[u]=8V(u)$ (cf. \eqref{Virial-Functional}.
\end{lemma}

\begin{lemma}\label{Virialzero}
Let $R\in [1, \infty]$, $\theta\in \R$ and $y\in\R$. Then we have
\[
F_{R}[e^{i\theta}P_{\omega}(\cdot-y)]=0.
\]
\end{lemma}

\begin{lemma}\label{VirialModulate11}
Let $u$ be the solution to \eqref{NLS} on an interval $I_{0}$. Let $R\in [1, \infty]$, $\chi: I_{0}\to \R$, $\theta: I_{0}\to \R$,
$y: I_{0}\to \R$. Then for all $t\in\R$,
\begin{align}\nonumber
	\frac{d}{d t}\P_{R}[u]&=F_{\infty}[u(t)]\\ \label{Modu11}
	                     &+F_{R}[u(t)]-F_{\infty}[u(t)]\\\label{Modu22}
											 &-\chi(t)\big\{F_{R}[e^{i\theta(t)}P_{\omega}(\cdot-y(t))]-F_{\infty}[e^{i\theta(t)}P_{\omega}(\cdot-y(t))]\big\}.
\end{align}

\end{lemma}

\section{Spectral properties of the linearized operator}\label{S:Spectral}
Let $u$ be a solution to \eqref{NLS}, and define $h=h_1+ih_2$ via
\[
h(t,x):=e^{-i \omega t}u(t,x)-P_{\omega}(x)
\]
Then, defining the operators $L_\pm$ (acting on $L^2(\R^3;\R)$) via
\begin{align*}
L_{+}h_{1}&=-\Delta h_{1}+\omega h_{1}-3P_{\omega}^{2}h_{1}+5P_{\omega}^{4}h_{1},\\
L_{-}h_{2}&=-\Delta h_{2}+\omega h_{2}-P_{\omega}^{2}h_{2}+P_{\omega}^{4}h_{2},	
\end{align*}
and setting
\begin{equation}\label{Resi}
\begin{split}
R(h)&=|h|^{2}h-h|h|^{4}+P_{\omega}\big( 2|h|^{2}+h^{2}-2|h|^{2}h^{2}-3|h|^{4}\big)\\
&-2P^{2}_{\omega}\big(\tfrac{1}{2}h^{3}+3|h|^{2}h +\tfrac{3}{2}|h|^{2}\overline{h}\big)
-2P^{3}_{\omega}\big( 2|h|^{2}+\tfrac{5}{2}h^{2}+\tfrac{1}{2}(\overline{h})^{2}\big),
\end{split}
\end{equation}
we find that $h$ satisfies the equation
\begin{equation}\label{Decomh}
\partial_{t}h+\L h=iR(h), \quad \text{where} \quad
\L:=\begin{pmatrix}
0 & -L_{-} \\
L_{+} & 0
\end{pmatrix}.
\end{equation}

The spectra of $L_{+}$ and $L_{-}$ consist of essential spectrum in $[\omega, \infty)$ and of a finite number of eigenvalues
in $(-\infty, \tilde{\omega}]$ for all $\tilde{\omega}<\omega$. Since $L_{-} P_{\omega}=0$ with $P_{\omega}>0$, it follows that
${\rm ker} \left\{L_{-}\right\}=\mbox{span}\left\{P_{\omega} \right\}$. In particular, by the Min-Max characterization of eigenvalues 
we have there exists $\eta>0$ such that
\begin{equation}\label{Lmenos}
\< L_{-} v, v \>\geq \eta \|v\|^{2}_{L^{2}}\quad \text{for $v\in H^{1}(\R^{3})$ with $(v, P_{\omega})_{L^{2}}=0$}.
\end{equation}

On the other hand, $L_{+}$ has only one negative eigenvalue $-\lambda_{1}$ with a corresponding  
eigenfunction $e_{1}\in H^{2}(\R^{3})$ (cf. \cite[Theorem 2.2]{KillipOhPoVi2017}). We assume that $\|e_{1}\|_{L^{2}}=1$. The second eigenvalue is $0$
and 
\[
{\rm ker} \left\{L_{+}\right\}=\mbox{span}\left\{\partial_{j}P_{\omega}: j=1,2, 3 \right\}.
\]
In particular, the space $H^{1}(\R^{3})$ can be decomposed into
\begin{equation}\label{Hdecom}
H^{1}(\R^{3})=\mbox{span}\left\{e_{1}\right\}\oplus\mbox{span}\left\{\partial_{j}P_{\omega}; j=1,2, 3 \right\}
\oplus E_{+},
\end{equation}
where  $L_{+}$ defines a positive definite quadratic form on $E_{+}$. Notice that in the direct sum \eqref{Hdecom} 
the spaces are mutually orthogonal with the inner product of $L^{2}(\R^{3})$.

We denote by $\F(g,h)$ the bilinear symmetric form 
\begin{equation}\label{Quadratic}
\F(g,h):=\tfrac{1}{2}\<L_{+}g_{1}, h_{1}  \>+\tfrac{1}{2}\<L_{-}g_{2}, h_{2}\>,
\end{equation}
where $g=g_1+ig_2$ and $h=h_1+ih_2$.  We denote $\F(h,h)$ by $\F(h)$, i.e.
\begin{equation}\label{LinearizedE}
\F(h)=\tfrac{1}{2}\int_{\R^{3}}|\nabla h|^{2}dx+\tfrac{1}{2}P^{4}_{\omega}(5h^{2}_{1}+h^{2}_{2})
-\tfrac{1}{2}P^{2}_{\omega}(3 h^{2}_{1}+h^{2}_{2})+\tfrac{\omega}{2}|h|^{2}\, dx.
\end{equation}

Note that
\begin{equation}\label{OrthoHB}
\F(i P_{\omega}, h)=\F(\partial_{j}P_{\omega}, h)=0\quad \text{for all $h\in H^{1}(\R^{3})$}.
\end{equation}

We have the following result.
\begin{lemma}\label{CoerQuadra}
There exists $C>0$ such that for every $h=h_{1}+ih_{2}\in H^{1}(\R^{3})$ satisfying
{  
\begin{align}\label{Qorto11}
&\<h_{1}, \partial_{1}P_{\omega}\>=\<h_{1}, \partial_{2}P_{\omega}\>=\<h_{1}, \partial_{3}P_{\omega}\>=\<h_{2}, P_{\omega}\>=0
\end{align}
and either 
\begin{align}
\label{Qorto22}
&\<h_{1}, L_{+}(x\cdot \nabla P_{\omega})\>=\<h_{1}, -2\Delta P_{\omega}\>=0;
\end{align}
or 
\begin{align}
\label{QortoX334}
&\<h_{1}, L_{+}(x\cdot \nabla P_{\omega})\>\geq 0;
\end{align}
or
\begin{align}
\label{QortoX33455}
&\<h_{1}, L_{+}(P_{\omega})\>\geq 0,
\end{align}
}
then we have
\[
\F(h)\geq C \|h\|^{2}_{H^{1}}.
\]
\end{lemma}
\begin{proof} We first prove that for all $f\in H^{1}(\R^{3})\backslash\left\{ 0\right\}$ satisfying
{  
\begin{equation}\label{LmasCon}
\tfrac{1}{2}\<f, L_{+}(x\cdot \nabla P_{\omega})\>=\<f, -\Delta P_{\omega}\>=\<f, \partial_{j}P_{\omega}\>=0 \quad \text{for all $j=1,2,3$,}
\end{equation}
}
we have
\begin{equation}\label{lem-step1}
\<L_{+}f,f\>>0.
\end{equation}
Indeed, setting $g_{\omega}:= \frac{1}{2} x\cdot \nabla P_{\omega}$, it follows from straightforward calculations (see \cite[Table 2.1]{KillipOhPoVi2017}) that 
\[
L_{+}g_{\omega}=-\Delta P_{\omega}\quad \text{and} \quad \<L_{+}g_{\omega}, g_{\omega}  \>=-\tfrac{1}{2}\|g_{\omega}\|^{2}_{\dot{H}^{1}}<0.
\]
Notice also that $(g_{\omega}, \partial_{j}P_{\omega})_{L^{2}}=0$ for all $j=1,2,3$. By the decomposition \eqref{Hdecom}, 
and \eqref{LmasCon} we infer that 
there exist $a$, $b\in \R$ and $\psi$, $\zeta\in E_{+}$ so that
\[
f=ae_{1}+\psi\qtq{and}g_{\omega}=be_{1}+\zeta.
\]
As $\<L_{+}g_{\omega}, g_{\omega}  \><0$, it follows that $b\neq 0$.

If $a=0$, then it is clear that  $\<L_{+}f,f\>=\<L_{+}\psi, \psi\>>0$ (recall that $f\neq 0$).  Suppose instead that $a\neq 0$. As $L_{+}$ defines a positive definite quadratic form on $E_{+}$, we have the Cauchy--Schwartz inequality 
\[
\<L_{+}\psi,\zeta\>^{2}\leq \<L_{+}\psi,\psi\>\<L_{+}\zeta,\zeta\>,
\]
which implies
\begin{equation}\label{AuxLe}
\<L_{+}f,f\>=-a^{2}\lambda_{1}+\<L_{+}\psi,\psi\>\geq 
-a^{2}\lambda_{1}+\frac{\<L_{+}\psi,\zeta\>^{2}}{\<L_{+}\zeta,\zeta\>}.
\end{equation}
Moreover, by using the fact that $L_{+}g_{\omega}=-\Delta P_{\omega}$
 and the orthogonality condition $\<f, \Delta P_{\omega}\>=0$, we have 
\[
0=-\<f, \Delta P_{\omega}\>=\<f, L_{+}g_{\omega}\>=-ab\lambda_{1}+\<L_{+}\psi,\zeta\>.
\]
Thus, $\<L_{+}\psi,\zeta\>=ab\lambda_{1}$, which implies
\begin{align*}
-a^{2}\lambda_{1}+\frac{\<L_{+}\psi,\zeta\>^{2}}{\<L_{+}\zeta,\zeta\>}
&=\frac{-a^{2}\lambda_{1}\<L_{+}\zeta,\zeta\>+a^{2}b^{2}\lambda^{2}_{1}}{\<L_{+}\zeta,\zeta\>}\\
&=	\frac{-a^{2}\lambda_{1}\<L_{+}g_{\omega}, g_{\omega}  \> }{\<L_{+}\zeta,\zeta\>}>0.
\end{align*}
Combined with \eqref{AuxLe}, this yields $\<L_{+}f, f \>>0$.

Next, we show that under conditions \eqref{LmasCon} there exists $C_{+}>0$ such that
\[
\<L_{+}f, f \>\geq C_{+}\|f\|^{2}_{H^{1}}.
\]
Suppose instead that there exists a sequence $\left\{f_{n}\right\}_{n\in \N}\subset H^{1}(\R^{3})$
such that $\<f_{n}, \Delta P_{\omega}\>=\<f_{n}, \partial_{j}P_{\omega}\>=0$ for all $j=1,2,3$,
\[
\|\nabla f_{n}\|^{2}_{L^{2}}+\omega \| f_{n}\|^{2}_{L^{2}}=1
\quad \text{and} \quad
\<L_{+}f_{n}, f_{n} \>\ \to 0 
\]
as $n\to \infty$. It is clear that $\left\{f_{n}\right\}_{n\in \N}$ is bounded in $ H^{1}(\R^{3})$. Therefore, there 
exists $f$ such that $f_{n}\rightharpoonup f$ in $H^{1}(\R^{3})$.  Notice that by the exponential decay of $P_{\omega}$ we get
as $n\to \infty$
\begin{align}\label{weakCondi11}
&5\int_{\R^{3}}	P^{4}_{\omega}f^{2}_{n}dx-3\int_{\R^{3}}P^{2}_{\omega}f^{2}_{n}dx\to  5\int_{\R^{3}}	P^{4}_{\omega}f^{2}dx-3\int_{\R^{3}}P^{2}_{\omega}f^{2}dx,\\ \label{weakCondi}
&\<f, \Delta P_{\omega}\>=\<f, \partial_{j}P_{\omega}\>=0 \quad \text{for all $j=1,2,3$}.
\end{align}
In particular, by weak lower semi-continuity of the $H^{1}(\R^{3})$- norm we obtain
\begin{equation}\label{weakZero}
\<L_{+}f,f\>\leq \liminf _{n\to \infty}\<L_{+}f_{n},f_{n}\>=0.
\end{equation}
Combing \eqref{weakCondi}, \eqref{weakZero}, and \eqref{lem-step1},  we see that $f\equiv 0$. However,
\[
3\int_{\R^{3}}P^{2}_{\omega}f^{2}-5\int_{\R^{3}}	P^{4}_{\omega}f^{2}
=1- \liminf _{n\to \infty}\<L_{+}f_{n},f_{n}\>=1,
\]
which thus yields a contradiction. Therefore, $\<L_{+}f, f \>\geq C_{+}\|f\|^{2}_{H^{1}}$ for some constant $C_{+}>0$.

On the other hand, by \eqref{Lmenos} and using an argument similar to the above we infer that there exists $C_{-}>0$ 
such that $\<L_{-}f, f \>\geq C_{-}\|f\|^{2}_{H^{1}}$. Therefore, for $h=h_{1}+ih_{2}\in H^{1}(\R^{3})$ we have
\[
\F(h)\geq \tfrac{C_{+}+C_{-}}{2}\|h\|^{2}_{H^{1}}.
\]
{ 
Next, assume that $\<h_{1}, L_{+}(x\cdot \nabla P_{\omega})\>\geq 0$. The proof follows using the same argument
(with some obvious modifications) used above.

Finally, assume that $\<h_{1}, L_{+}(P_{\omega})\>\geq 0$.  In this case, setting $g_{\omega}:=P_{\omega}$ we have 
(see \cite[Table 2.1]{KillipOhPoVi2017})
\[
L_{+}g_{\omega}=4P^{5}_{\omega}-2P^{3}_{\omega},
\quad \quad \<L_{+}g_{\omega}, g_{\omega}  \>=\tfrac{4}{3}(\beta(\omega)-2)\|P_{\omega}\|^{2}_{\dot{H}^{1}}<0,
\]
and $(g_{\omega}, \partial_{j}P_{\omega})_{L^{2}}=0$ for all $j=1,2,3$. Now, by applying the same argument as above
to $g_{\omega}=P_{\omega}$, we find that $\F(h)\geq C \|h\|^{2}_{H^{1}}$ holds for all $f\in H^{1}(\R^{3})$.}
\end{proof}

{ 
\begin{corollary}\label{ConHPG}
There exists $C>0$ such that for every $h=h_{1}+ih_{2}\in H^{1}(\R^{3})$ satisfying
\begin{align}\label{CoroC11}
&\<h_{1}, \partial_{1}P_{\omega}\>=\<h_{1}, \partial_{2}P_{\omega}\>=\<h_{1}, \partial_{3}P_{\omega}\>=\<h_{2}, P_{\omega}\>=0
\end{align}
and
\begin{align}\label{CoroC22}
&\<h_{1}, L_{+}(x\cdot \nabla P_{\omega}+\tfrac{3}{2}P_{\omega})\>=0,
\end{align}
we have
\[
\F(h)\geq C \|h\|^{2}_{H^{1}}.
\]
\end{corollary}
\begin{proof}
Consider $h=h_{1}+ih_{2}$. By \eqref{CoroC22} we have
\begin{align}\label{UCCo}
	&\<h_{1}, L_{+}(x\cdot \nabla P_{\omega})\>+\<h_{1}, L_{+}(\tfrac{3}{2}P_{\omega})\>=0.
\end{align}
Now, if $\<h_{1}, L_{+}(x\cdot \nabla P_{\omega})\>\geq 0$, then \eqref{Qorto11} (recall \eqref{CoroC11}) and \eqref{QortoX334} hold. Lemma~\ref{CoerQuadra} implies that
$\F(h)\geq C \|h\|^{2}_{H^{1}}$ for some constant $C>0$.

On the other hand, if $\<h_{1}, L_{+}(x\cdot \nabla P_{\omega})\>< 0$, then by \eqref{UCCo} we infer that 
$\<h_{1}, L_{+}(\tfrac{3}{2}P_{\omega})\>\geq 0$, which implies that \eqref{Qorto11}  and \eqref{QortoX33455} hold.  
Again, by Lemma~\ref{CoerQuadra} we find $\F(h)\geq  \tilde{C} \|h\|^{2}_{H^{1}}$ for some constant $\tilde{C}>0$.
\end{proof}
}

The following lemma gives the general structure of the spectrum of the operator $\L$.  The proof of this result is given in Appendix~\ref{S:A}.
\begin{lemma}\label{SpecLL}
The operator $\L$ defined on $L^{2}(\R^{3})\times L^{2}(\R^{3})$ with domain $H^{2}(\R^{3})\times H^{2}(\R^{3})$
has two simple eigenfunctions $e_{+}$ and $e_{-}=\overline{e_{+}}$ in $\Sch(\R^{3})$ with real eigenvalues
$\pm \lambda_{1}$ with $\lambda_{1}>0$. Moreover
\[
\sigma(\L)\cap \R=\left\{-\lambda_{1}, 0, \lambda_{1}\right\}.
\]
The essential spectrum of $\L$ is 
$\left\{i\xi: \xi\in \R, |\xi|\geq \omega\right\}$, and the kernel is
\[
\mbox{ker}\left\{\L\right\}
=\mbox{span}\left\{\partial_{1}P_{\omega}, \partial_{2}P_{\omega}, \partial_{3}P_{\omega}, iP_{\omega}  \right\}.
\]
\end{lemma}

We set 
\[
e_{1}:=\RE e_{+}\quad \text{and}\quad e_{2}:= \IM e_{+},
\]
so that that $L_{+}e_{1}=\lambda_{1}e_{2}$ and $L_{-}e_{2}=-\lambda_{1}e_{1}$.

The following lemma shows that the linearized energy $\F$ is coercive in the set $Y^{\bot}$ of $h\in H^{1}(\R^{3})$ satisfying the orthogonality relations \eqref{NewCon} and \eqref{Newcon22}.
\begin{lemma}\label{NewCoer}
There exists $C>0$ such that for every $h=h_{1}+ih_{2}\in H^{1}(\R^{3})$ satisfying
\begin{align}\label{NewCon}
&\<h_{1}, \partial_{1}P_{\omega}\>=\<h_{1}, \partial_{2}P_{\omega}\>
=\<h_{1}, \partial_{3}P_{\omega}\>=\<h_{2}, P_{\omega}\>=0,\\
\label{Newcon22}
&\<e_{1}, h_{2}\>=\<e_{2}, h_{1}\>=0,
\end{align}
we have
\[
\F(h)\geq C \|h\|^{2}_{H^{1}}.
\]
\end{lemma}
\begin{proof} The proof is similar to that of  \cite[Proposition 2.7]{DuyckaertsRou2010}.

We first show that if $h=h_{1}+ih_{2}\in H^{1}(\R^{3})\backslash\left\{0\right\}$ satisfies \eqref{NewCon} and  \eqref{Newcon22}, then $\F(h)>0$. Suppose instead that there exists nonzero $g \in H^{1}(\R^{3})$ such that
\begin{equation}\label{Ortog}
\F(g)\leq 0,\quad 
\<e_{1}, g_{2}\>=\<e_{2}, g_{1}\>=\< g_{2}, P_{\omega}\>=\<g_{1}, \partial_{j}P_{\omega}\>=0,
\quad j=1,2,3.
\end{equation}
We now note that (cf. Remarks~\ref{PLf}~and~\ref{PLf11})
\begin{equation}\label{FEE}
\F(e_{+}, g)=0, \quad \F(e_{-}, g)=0, \quad \F(e_{+}, e_{+})=0
\quad \text{and}\quad 
\F(e_{+}, e_{-})\neq 0.
\end{equation}
We also recall that
\begin{equation}\label{recallH}
\F(iP_{\omega})=\F(\partial_{j}P_{\omega})=0 \quad \text{for all $j=1,2,3$}.
\end{equation}
Now we set
\[
E_{-}:=\mbox{span}\left\{\partial_{1}P_{\omega},  \partial_{2}P_{\omega},
\partial_{3}P_{\omega}, iP_{\omega}, e_{+}, g\right\}
\]
Combining \eqref{Ortog}, \eqref{FEE} and \eqref{recallH}  we get $\F(h)\leq 0$ for all $h\in E_{-}$.
Similarly, by using \eqref{Ortog}, \eqref{FEE} and \eqref{recallH} one can show that $\mbox{dim}_{\R} E_{-}=6$ (see \cite[Proposition 2.7]{DuyckaertsRou2010} for more details). However, Lemma~\ref{CoerQuadra} shows that $\F$ is definite positive on a codimension $5$ subspace of $H^{1}(\R^{3})$, which is a contradiction. 

Finally, the same argument given in the second part of the proof of Lemma~\ref{CoerQuadra} shows that $\F(h)\geq C \|h\|^{2}_{H^{1}}$ for all $h\in H^{1}(\R^{3})$.
\end{proof}

\begin{remark}\label{PLf} A direct computation shows that for any $f$, $g\in H^{1}$,
\begin{align*}
&\F(e_{\pm})=0, \quad \F(P_{\omega})=\tfrac{4}{3}[\beta(\omega)-2]\|\nabla P_{\omega}\|^{2}_{L^{2}},\\
&\F(h,g)=\F(g,h), \quad \F(\L h, g)=-\F(h, \L g)
\end{align*}
\end{remark}

\begin{remark}\label{PLf11} We have that $\F(e_{+}, e_{-})\neq 0$. Indeed, suppose instead that $\F(e_{+}, e_{-})= 0$. Consider 
\[h\in \mbox{span}\left\{iP_{\omega}, e_{+}, e_{-}, \partial_{1}P_{\omega}, \partial_{2}P_{\omega}, \partial_{3}P_{\omega}\right\},\] 
which is of codimension $6$.  Then by Remark~\ref{PLf} we see that $\F(h)=0$, which is a contradiction because $\F$ is positive on a codimension 5 subspace (see Proposition~\ref{CoerQuadra}).
\end{remark}

\begin{remark}\label{PE1}
We have
\[
\int_{\R^{3}}\Delta P_{\omega} e_{1}dx\neq 0.
\]
Indeed, suppose instead that $\langle \Delta P_{\omega}, e_{1} \rangle=0$. Recall that
\[
L_{+}g_{\omega}=-\Delta P_{\omega}\quad \text{and} \quad \<L_{+}g_{\omega}, g_{\omega}  \>=-\tfrac{1}{2}\|g_{\omega}\|^{2}_{\dot{H}^{1}}<0,
\]
where $g_{\omega}:= \frac{1}{2} x\cdot \nabla P_{\omega}$.  Notice that 
$\langle L_{+}(g_{\omega}), e_{1} \rangle=\lambda_{1}\langle g_{\omega}, e_{2} \rangle$. On the other hand,  $\langle L_{+}(g_{\omega}), e_{1} \rangle=\langle -\Delta P_{\omega}, e_{1} \rangle=0$. Therefore, $\langle g_{\omega}, e_{2} \rangle=0$. Then, by Lemma~\ref{NewCoer} we infer that $\F(g_{\omega})>0$, contradicting that $\F(g_{\omega})=\<L_{+}g_{\omega}, g_{\omega}  \><0$. 
\end{remark}

Finally, we record the following identity, which follows from direct computation:
\begin{lemma}\label{Fuu}
Let $h\in H^{1}(\R^{3})$ and assume $E(P_{\omega}+h)=E(P_{\omega})$ and $M(P_{\omega}+h)=M(P_{\omega})$.
Then
\begin{align*}
\F(h)&=\int_{\R^{3}}P_{\omega}(|h|^{2}h_{1}-|h|^{4}h_{1})dx -\tfrac{1}{2}\int_{\R^{3}}P^{2}_{\omega}(|h|^{4}+4|h|^{2}h^{2}_{1})dx\\
&\quad -\tfrac{1}{3}\int_{\R^{3}}P^{3}_{\omega}(6|h|^{2}h_{1}+4h^{3}_{1})dx +\tfrac{1}{4}\int_{\R^{3}}|h|^{4}dx-\tfrac{1}{6}\int_{\R^{3}}|h|^{6}dx.
\end{align*}
\end{lemma}

\section{Modulation analysis}\label{S:Modula}

Throughout this section, we fix $(m, e)=(M(P_\omega),E(P_\omega))\in \partial\K_{s}$ and a solution $u$ to \eqref{NLS} satisfying
\[
(M(u),E(u))=(m,e) \qtq{and} V(u_0)>0.
\]
We denote 
\[
\delta(t):=V(u(t)) \quad \text{for $t\in \R$},
\]
where $V$ is the virial functional \eqref{Virial-Functional}. By Lemma~\ref{PositiveV} we have 
\begin{equation}\label{deltapositive}
\delta(t)=V(u(t))>0, \quad \text{for all $t\in \R$}.
\end{equation}

We let $\delta_{0}>0$ be a small parameter and define the open set
\[
I_{0}=\left\{t\in [0, \infty):\delta(t)<\delta_{0}\right\}.
\] 
We will prove the following. 

\begin{proposition}\label{ModilationFree}
For $\delta_{0}>0$ sufficiently small, there exist functions $\theta: I_{0}\to \R$, 
$\alpha: I_{0}\to \R$, $y: I_{0}\to \R^{3}$, and $h:I_0\to H^1$ such that we can write
\begin{equation}\label{DecomUFree}
e^{-i\omega t}u(t, \cdot+y(t))=e^{i \theta(t)}[(1+\alpha(t))P_{\omega}+h(t)]\qtq{for all}t\in I_0,
\end{equation}
with
\begin{equation}\label{EstimateFree}
|y^{\prime}(t)|+|\theta^{\prime}(t)|+|\alpha^{\prime}(t)|
\lesssim\delta(t)\sim |\alpha(t)|\sim \|h(t)\|_{H^{1}}.
\end{equation}
\end{proposition}

Using the implicit theorem and the variational characterization of $P_{\omega}$ in Lemma~\ref{PriModula}, we can obtain the following orthogonal decomposition.  We will omit the proof, as it is essentially the same given in \cite[Lemma 4.1]{DuyckaertsRou2010}.

\begin{lemma}\label{ExistenceFree}
 There exist $\delta_{0}>0$, a positive function $\epsilon(\delta)$ defined for $0<\delta\leq \delta_{0}$ and functions $\sigma: I_{0}\to \R$ and $y: I_{0}\to \R^{3}$ such that if
$\delta(t)<\delta_{0}$, then
\begin{equation}\label{TaylorF}
\| u(t)-e^{i\sigma(t)}P_{\omega}(\cdot-y(t))\|_{H^{1}}\leq \epsilon(\delta).
\end{equation}
The mapping $u\mapsto (\sigma, y)$ is $C^{1}$ and $\epsilon(\delta)\to 0$ as $\delta\to 0$. The functions $\sigma(\cdot)$ and $y(\cdot)$ are chosen to impose the following orthogonality conditions: for $j\in\{1,2,3\}$, 
\begin{equation}\label{OrtFree}
\IM\<e^{i\sigma(t)}P_{\omega}(\cdot-y(t)), u(t)\>=\RE\<e^{i\sigma(t)}\partial_{j}P_{\omega}(\cdot-y(t)), u(t)\>=0.
\end{equation}
\end{lemma}

With $(\sigma(t),y(t))$ as in Lemma~\ref{ExistenceFree}, we write
\begin{equation}\label{DefH}
e^{-i\theta(t)}[e^{-i \omega t}u(t, x+y(t))]=(1+\alpha(t))P_{\omega}(x)+h(t,x),
\end{equation}
where $\theta(t):=\sigma(t)-\omega t$ and
{
\[
\alpha(t)=\frac{\<L_{+}(x\cdot\nabla P_{\omega}+\tfrac{3}{2} P_{\omega}),  e^{-i\theta(t)}e^{-i \omega t} u(t, \cdot+y(t))\>}{\<L_{+}(x\cdot\nabla P_{\omega}+\tfrac{3}{2} P_{\omega}), P_{\omega}\>}-1.
\]
We then have the following orthogonally conditions for $h$: for $j\in\{1,2,3\}$, 
\begin{equation}\label{h: Cond}
\IM\<h(t), P_{\omega}\>=\RE\< h(t),L_{+}(x\cdot\nabla P_{\omega}+\tfrac{3}{2} P_{\omega})\>=\RE\<h(t), \partial_{j}P_{\omega}\>\equiv0.
\end{equation}
We also note that $\<L_{+}(x\cdot\nabla P_{\omega}+\tfrac{3}{2} P_{\omega}), P_{\omega}\>=2[\beta(\omega)-1]\|\nabla P_{\omega}\|^{2}_{L^{2}}$.
}

The following lemma relates the parameters $\delta(t)$, $\alpha(t)$, and $h(t)$. 
\begin{lemma}\label{BoundI}
For all $t\in I_0$, 
\begin{equation}\label{DeltaBound}
\delta(t)\sim |\alpha(t)|\sim \|h(t)\|_{H^{1}}.
\end{equation}
\end{lemma}
\begin{proof}
{
We expand the virial functional around $P_{\omega}$ (recalling that $V(P_{\omega})=0$) to obtain
\begin{align*}
\delta(t)&=	V(P_{\omega}+\alpha P_{\omega}+h)-V(P_{\omega})\\
&=\<-2\Delta P_{\omega}+ 6P_{\omega}^{5}-3P_{\omega}^{3},\,\, \alpha(t) P_{\omega}+h_{1}(t)  \>+O({\alpha}^{2}+\|h\|^{2}_{H^{1}}) \\
&=\<L_{+}(x\cdot\nabla P_{\omega}+\tfrac{3}{2} P_{\omega}),\,\, \alpha(t) P_{\omega}+h_{1}(t)  \>+O({\alpha}^{2}+\|h\|^{2}_{H^{1}}),
\end{align*}
where we have used $-2\Delta P_{\omega}+ 6P_{\omega}^{5}-3P_{\omega}^{3}=L_{+}(x\cdot\nabla P_{\omega}+\tfrac{3}{2} P_{\omega})$.
By \eqref{h: Cond} we infer that
\begin{align}\label{DeltaAlfa}
&\delta(t)=\<L_{+}(x\cdot\nabla P_{\omega}+\tfrac{3}{2} P_{\omega}), P_{\omega}\>\alpha(t)+O({\alpha}^{2}+\|h\|^{2}_{H^{1}}).
\end{align}
On the other hand, by using Lemma~\ref{Fuu} we can write (recalling that $\F(P_{\omega})<0$, cf. Remark~\ref{PLf}):
\begin{equation}\label{QuadraExp}
\F(h)=\alpha^{2}[-\F(P_{\omega})]+2\alpha [-\F(P_{\omega}, h_{1})]+O(\alpha^{3}+\|h\|^{3}_{H^{1}}).
\end{equation}
In particular, as $\F(h) \lesssim \|h\|^{2}_{H^{1}}$ we see that $|\alpha|=O(\|h\|_{H^{1}})$. Moreover, by Corollary~\ref{ConHPG} (recalling the orthogonally conditions for $h$, \eqref{h: Cond})
and \eqref{QuadraExp} we infer that  $\|h\|_{H^{1}}=O(|\alpha|)$, i.e., $\|h\|_{H^{1}}\sim |\alpha|$. 
 
Finally, from \eqref{DeltaAlfa} we see that
\[
\delta(t)=\<L_{+}(x\cdot\nabla P_{\omega}+\tfrac{3}{2} P_{\omega}), P_{\omega}\>\alpha(t)+O({\alpha}^{2}),
\]
which implies that $\delta \sim |\alpha|$. Therefore,  $\|h\|_{H^{1}}\sim |\alpha|\sim \delta$.}\end{proof}

\begin{lemma}\label{BoundII}
 Let $(\sigma(t), y(t))$ be as in Lemma~\ref{ExistenceFree}  and $h(t)$, $\theta(t)$ and $\alpha(t)$ be as in \eqref{DefH}. Then, taking a smaller $\delta_{0}$ if necessary, we have
\[
|y^{\prime}(t)|+|\theta^{\prime}(t)|+|\alpha^{\prime}(t)|
\lesssim\delta(t)
\]
\end{lemma}

\begin{proof}
With Lemma~\ref{BoundI} in hand, the proof is very similar to that of \cite[Lemma~4.3]{DuyckaertsRou2010} (in particular, the key idea is to differentiate the orthogonality relations).  We omit the details. 
\end{proof}

Proposition~\ref{ModilationFree} now follows immediately from Lemmas~\ref{ExistenceFree}, \ref{BoundI} and \ref{BoundII}.

To close this section, we record one final lemma, which shows how control over $\delta(t)$ guarantees exponential closeness to the orbit of $P_\omega$.  

\begin{lemma}\label{Decaiment}
Suppose that there exist $a$, $b>0$
so that
\begin{equation}\label{Decau}
\int^{\infty}_{t}\delta(s)\,ds\leq a e^{-b t} \quad \text{for all $t\geq 0$}.
\end{equation}
Then $\lim_{t\to \infty}\delta(t)=0$. Moreover, there exist $\theta_{0}\in \R$, $y_{0}\in \R^{3}$ and $c$, $C>0$ such that
\[
\|u(t)-e^{i\theta_{0}}e^{i \omega t}P_{\omega}(\cdot-y_{0})\|_{H^{1}}\leq C e^{-c t}.
\]
\end{lemma}
\begin{proof}
With Proposition~\ref{ModilationFree} in hand, the proof  is essentially the same as that of \cite[Lemma~4.4]{DuyckaertsRou2010}.\end{proof}

\section{Concentration compactness}\label{S: Concentrated}

In this section, we show that the solutions to \eqref{NLS} with $(M(u), E(u))=(M(P_\omega),E(P_\omega))\in\partial \K_s$ that do not scatter forward in time, must converge exponentially to $P_{\omega}$ as $t \to \infty$ up to
the symmetries of the equation.

\begin{proposition}\label{CompacDeca}
Let $u$ be a solution of \eqref{NLS} satisfying 
\begin{align}\label{PropCon11}
&(M(u), E(u))=(M(P_\omega),E(P_\omega))\in\partial\K_s \quad \text{with $V(u_{0})>0$},\\
\label{PropCon22}
&\|u\|_{L^{10}_{t, x}((0, \infty)\times \R^{3})}=\infty.
\end{align}
Then there exist $\theta\in \R$, $y_{0}\in \R^{3}$ and $c$, $C>0$ such that
\[
\|u(t)-e^{i\theta}e^{i \omega t}P_{\omega}(\cdot-y_{0})\|_{H^{1}}\leq Ce^{-c t}.
\]
\end{proposition}

As a corollary, we will obtain the following: 
\begin{corollary}\label{ClassC}
There is no solution to \eqref{NLS} satisfying \eqref{PropCon11} and 
\begin{equation}\label{Infinity10}
\|u\|_{L^{10}_{t, x}((0, \infty)\times \R^{3})}=
\|u\|_{L^{10}_{t, x}((-\infty,0)\times \R^{3})}=
\infty.
\end{equation}
\end{corollary}

The first step in the proof of Proposition~\ref{CompacDeca} is to establish compactness for nonscattering solutions. 

\begin{lemma}\label{Compacness11}
Suppose that $u(t)$ satisfies \eqref{PropCon11} and \eqref{PropCon22}. Then there exists 
$x_{0}:[0, \infty) \to \R^{3}$ such that 
 \begin{equation}\label{CompactX}
\left\{u(t, \cdot+x_{0}(t)): t\in [0, \infty)\right\} \quad
\text{is pre-compact in $H^{1}(\R^{3})$}.
\end{equation}
\end{lemma}
\begin{proof} First note that by Lemma~\ref{PositiveV}, the solution $u(t)$ is uniformly bounded in $H^{1}(\R^{3})$ and $V(u(t))>0$ for all $t\in \R$.

Given $\T_{n}\to \infty$, we will first show that there exists a subsequence in $n$ and $x_{n}\in\R^3$ such that
\begin{equation}\label{SequenceC}
\text{$u(\T_{n}, x+x_{n})$ converges in $H^{1}(\R^{3})$. }
\end{equation}
We apply the profile decomposition Proposition~\ref{ProfL} to the sequence $u(\T_{n})$ to obtain
\begin{equation}\label{ProfileB}
u_{n}:=u(\T_{n})=\sum^{J}_{j=1}\psi_{n}^{j}+W^{J}_{n}, 
\end{equation}
with all of the properties stated in Proposition~\ref{ProfL}.

If $J^{\ast}=0$, then
\[
\|e^{it\Delta}u(\T_{n})\|_{L^{10}_{t,x}([0, \infty)\times \R^{3})} \to 0 \quad \text{as $n\to \infty$}.
\]
In particular, by using the stability result (Lemma~\ref{S: result}) with $\tilde{u}_{n}=e^{it\Delta}u(\T_{n})$ and $u_{n}=u(\T_{n})$, we obtain
\[
\|u(t+\T_{n})\|_{L^{10}_{t,x}((0, \infty)\times \R^{3})}=
\|u\|_{L^{10}_{t,x}((\T_{n}, \infty)\times \R^{3})}\lesssim 1,
\]
for large $n$, which is a contradiction. Therefore $J^{\ast}\geq 1$. 

Now suppose that $J^{\ast}\geq 2$. By decoupling of mass and energy, we have
\begin{align}\label{Pyt11}
&\lim_{n\rightarrow\infty}\big[\sum^{J}_{j=1}M(\psi_{n}^{j})+M(W^{J}_{n})\big]
=\lim_{n\rightarrow\infty}M(u_{n})=M(P_{\omega}),\\\label{Pyt22}
&\lim_{n\rightarrow\infty}\big[\sum^{J}_{j=1}E(\psi_{n}^{j})+E(W^{J}_{n})\big]
=\lim_{n\rightarrow\infty}E(u_{n})=E(P_{\omega}).
\end{align}
for any $0\leq J\leq J^{\ast}$.


By assumption, we have $E(P_{\omega})>0$. We claim that there exists $\delta>0$ such that for all $n$ large and $1\leq j \leq J$,
\begin{equation}\label{MEdelta}
M(\psi_{n}^{j})\leq M(P_{\omega})- \delta
\quad \text{and}\quad 
E(\psi_{n}^{j})\leq E(P_{\omega})- \delta.
\end{equation}
In particular, by the strict monotonicity of $m\mapsto \EE(m)$,  \eqref{MEdelta} implies that $(M(\psi_{n}^{j}), E(\psi_{n}^{j}))\in \K$ (the scattering region) for all $n$ large and $1\leq j \leq J$.

\begin{proof}[Proof of \eqref{MEdelta}] First, suppose that $\lambda^{j}_{n}\equiv 1$. Then there exists $\delta>0$ such that $M(\psi_{n}^{j})\leq M(P_{\omega})- \delta$. In particular, as  $M(P_{\omega})\leq M(Q_{1})$, by \eqref{BoundH} we obtain
\[
E(\psi_{n}^{j})
\gtrsim_{\delta}\|\psi^{j}\|^{2}_{\dot{H}^{1}},
\]
which implies \eqref{MEdelta}. Now suppose instead $\lambda^{j}_{n} \to 0$ as $n\to \infty$.  By Bernstein's inequality and a changes of variables we see that $M(\psi_{n}^{j})\lesssim (\lambda^{j}_{n})^{1-\theta}\|\psi^{j}\|_{\dot{H}^{1}}$. As $\lambda^{j}_{n} \to 0$, we find that there exists $\delta>0$ such that $M(\psi_{n}^{j})\leq M(P_{\omega})- \delta$ for $n$ large.
Moreover, using 
\[
\|\psi_{n}^{j}\|_{\dot{H}^{1}}=\|P_{\geq (\lambda^{j}_{n} )^{\theta}}\psi^{j}\|_{\dot{H}^{1}}\to
 \|\psi^{j}\|_{\dot{H}^{1}}\quad \text{as $n\to \infty$}
\]
together with \eqref{BoundH}, we have that 
\[
\liminf_{n\to \infty}E(\psi_{n}^{j})
\gtrsim_{\delta}\liminf_{n\to \infty}\|\psi_{n}^{j}\|^{2}_{\dot{H}^{1}}
\gtrsim \|\psi^{j}\|^{2}_{\dot{H}^{1}}
\quad \text{for  $1\leq j \leq J$},
\]
which yields \eqref{MEdelta}. \end{proof}

Now we define solutions $v^{j}$ to \eqref{NLS} as follows:
\begin{enumerate}[label=\rm{(\roman*)}]
\item If $\lambda^{j}_{n}\equiv 1$ and $t^{n}_{j}\rightarrow\pm\infty$ as $n\rightarrow\infty$, we define $v^{j}$  to be the  global solution to \eqref{NLS} which scatters to $e^{it\Delta}\psi^{j}$ when $t\rightarrow \pm\infty$. Note that
\[
\lim_{n\to \infty}\|v^{j}(t^{n}_{j})-\psi_{n}^{j}\|_{H^{1}}=0.
\]
In particular, since \eqref{MEdelta} holds, by  Theorem~\ref{MainTheorem} we obtain  global space-time bounds for the solution $v^{j}$.

\item If $\lambda^{j}_{n}\equiv 1$ and $t^{n}_{j}\equiv 0$, we define $v^{j}$ to be the global solution to \eqref{NLS}
	with the initial data $v^{j}(0)=\psi^{j}$. Since 
$(M(\psi^{j}), E(\psi^{j}))\in \K$ (see \eqref{MEdelta}), Theorem~\ref{MainTheorem} implies that the solution $v^{j}$ is a global solution
with finite scattering size.
	
\item If $\lambda^{j}_{n}\rightarrow 0$ as $n\rightarrow\infty$, we define $v_{n}^{j}$ to be the solution to \eqref{NLS}
	with the initial data $v_{n}^{j}(0)=\psi_{n}^{j}$  established in \cite[Proposition 8.3]{KillipOhPoVi2017}.
\end{enumerate}

In the cases (i) and (ii) we define the global solutions to \eqref{NLS},
\[v^{j}_{n}(t,x):=v^{j}(t+t^{j}_{n}, x-x^{j}_{n}).\]

By \eqref{MEdelta} and  inequality \eqref{BoundH} we get $\|v^{j}_{n}(t)\|^{2}_{\dot{H}_{x}^{1}}\lesssim_{\delta} E(v^{j}_{n})$. Thus,
persistence of regularity (see Lemma 6.2 in \cite{KillipOhPoVi2017}) implies
\begin{equation}\label{Boundcriti}
\|v^{j}_{n}\|_{L^{10}_{t, x}(\R\times\R^{3})}
+\|\nabla v^{j}_{n}\|_{L^{\frac{10}{3}}_{t, x}(\R\times\R^{3})}\lesssim_{\delta} E(v^{j}_{n})^{\frac{1}{2}}.
\end{equation}
Similarly,  we get
\begin{equation}\label{Masscriti}
\|v^{j}_{n}\|_{L^{\frac{10}{3}}_{t, x}(\R\times\R^{3})}\lesssim_{\delta} M(v^{j}_{n})^{\frac{1}{2}}.
\end{equation}
Moreover, an argument similar to the one above shows that
 \begin{equation}\label{W: Bound}
\|W_{n}^{J}\|^{2}_{\dot{H}^{1}}\lesssim_{\delta} E(W^{J}_{n}).
\end{equation}

We now define approximate solutions to \eqref{NLS} by
\[u^{J}_{n}(t,x)= \sum^{J}_{j=1}v^{j}_{n}(t,x)+e^{it\Delta}W^{J}_{n},\]
and we use perturbation argument to obtain a contradiction to \eqref{PropCon22}.  We will verify that for $n$ and $J$ large,  $u_{n}^{J}$
is an approximate solution to \eqref{NLS} with global finite space-time bounds. Indeed, 
with estimates \eqref{Boundcriti}, \eqref{Masscriti} and \eqref{W: Bound} in hand, we can repeat the argument of 
\cite[Lemmas 9.3, 9.4 and 9.5]{KillipOhPoVi2017} to show that 
\begin{align}
&\lim_{n\to \infty}\|u^{J}_{n}(0)-u_{n}(0)\|_{H^{1}_{x}}=0,\\
&\sup_{J}\limsup_{n\rightarrow\infty}\left[ \| {u}^{J}_{n}  \|_{L^{10}_{t,x}(\R \times \R^{3})} 
+\| {u}^{J}_{n}  \|_{L_{t}^{\infty}H^{1}_{x}(\R\times \R^{3})} \right]\lesssim_{\tau_{c},\delta} 1,\\
&\lim_{J\to J^{\ast}}\limsup_{n\rightarrow\infty}
\|\nabla [i\partial_{t}u^{J}_{n}+\Delta u^{J}_{n}-F(u^{J}_{n})] \|_{L^{\frac{10}{7}}_{x}(\R\times \R^{3})}=0,
\end{align}
where $F(z)=|z|^{4}z-|z|^{2}z$. Thus, an application of the stability result (Lemma~\ref{S: result})
then yields that $u$ satisfies finite space-time bounds globally in time, i.e. $u\in L^{10}_{t,x}(\R \times \R^{3})$, 
contradicting \eqref{PropCon22}  and so $J^{\ast}\geq 2$ cannot occur.

From the analysis above we obtain that $J^{\ast}=1$. Therefore,
\[
u_{n}=u(\T_{n})=\psi_{n}+W_{n}
\]
with $W_{n} \rightharpoonup 0$ weakly in $\dot{H}^{1}(\R^{3})$ (cf. \eqref{WeakConver}).
To prove the proposition we need to show that $W_{n} \to 0$ strongly in $H^{1}(\R^{3})$,  
$\lambda_{n}\equiv 1$ and $t_{n} \equiv 0$. First we show that $W_{n} \to 0$ in $H^{1}(\R^{3})$.
Recall that $E(P_{\omega})>0$ by assumption. If $W_{n} \nrightarrow 0$  in $\dot{H}^{1}(\R^{3})$, as $M(W_{n})\leq M(P_{\omega})<M(Q_{1})$, then by inequality \eqref{BoundH} we infer that
$\liminf_{n \to \infty}E(W_{n})>0 $ and one can show by the previous argument that $u$ scatters in $H^{1}(\R^{3})$.
Therefore, $W_{n} \to 0$ in $\dot{H}^{1}(\R^{3})$. In particular, $\lim_{n\to \infty}E(W_{n})=0$ and 
$\lim_{n\to \infty}E(\psi_{n})=E(P_{\omega})$ (see \eqref{Pyt22}). Moreover, if $M(\psi_{n})\leq M(P_{\omega})-\delta$ for some $\delta>0$ and $n$ large, then since $m\mapsto \EE(m)$ is strictly decreasing we infer that (again by the above argument) $u$ scatters in $H^{1}(\R^{3})$, contradicting \eqref{PropCon22}.
Thus, $\lim_{n\to \infty}M(\psi_{n})=M(P_{\omega})$ and $\lim_{n\to \infty}M(W_{n})=0$. 


If $\lambda_{n}\to 0$ as $n\to \infty$, by using Proposition 8.3 in \cite{KillipOhPoVi2017} we see that the unique global solution 
with initial data $v_{n}(0)=\psi_{n}$ satisfies
$\| v_{n}\|_{L^{10}_{t, x}(\R\times \R^{3})}\leq C(\|\psi\|)$ for $n$ sufficiently large.
As
\[
\lim_{n\to \infty}\|u_{n}-v_{n}(0)\|_{H^{1}}=\lim_{n\to \infty}\|W_{n}\|_{H^{1}}=0,
\]
applying the stability result Lemma~\ref{S: result} we obtain a contradiction to \eqref{PropCon22}. 

Now we preclude the possibility that $t_{n}\to \pm\infty$. If $t_{n}\to \infty$ as $n\to \infty$ (say), then by Strichartz estimates, monotone convergence  theorem  we infer that
\[
\begin{split}
\| e^{i\Delta}u_{n}  \|_{L^{10}_{t,x}([0,\infty)\times \R^{3})}\leq \|e^{it \Delta} \psi_{n}  \|_{L^{10}_{t,x}([0,\infty)\times \R^{3})}+
\|e^{it \Delta} W_{n} \|_{L^{10}_{t,x}([0,\infty)\times \R^{3})},\\
\lesssim
\|e^{it \Delta}  \phi \|_{L^{10}_{t,x}([t_{n},\infty)\times \R^{3})}+\|W_{n}\|_{H^{1}}\rightarrow0.
\end{split}
\]
Applying the stability result once more,  we reach a contradiction.

Finally, by \eqref{SequenceC} and using the same argument developed in \cite{DuyHolmerRoude2008} we deduce that there exist a function
$x_{0}(t)$ such that \eqref{CompactX} holds, which finally completes the proof of the proposition.\end{proof}

Next, the same argument as in \cite[Lemma 4.2]{MiaMurphyZheng2021} yields the following lemma.

\begin{lemma}\label{Parametrization}
If $\delta_{0}>0$ is sufficiently small, then there exists $C>0$ such that
\[
|x_{0}(t)-y(t)|<C \quad \text{for  $t\in I_{0}$},
\]
where the parameter $y(t)$ is given in Proposition~\ref{ModilationFree}.
\end{lemma}

Using Lemma~\ref{Parametrization}, we can define the modified spatial center $x(t)$ via
\begin{equation}\label{Def:x}
x(t)=
\begin{cases}
x_{0}(t)& \quad t\in [0, \infty)\setminus I_{0},\\
y(t)&\quad t\in I_{0}.
\end{cases}
\end{equation}
In this case, we still have that \begin{equation}\label{CompacNew}
\text{$\left\{u(t, \cdot+x(t))\right\}$ is pre-compact in $H^{1}(\R)$}.
\end{equation}

Next, we use the minimality property of $\partial\K$ to obtain control over the motion of the spatial center $x(t)$:

\begin{lemma}\label{Mome}
Let $u(t)$ be the solution in Lemma~\ref{Compacness11} with $x(t)$ defined in \eqref{Def:x}.
Then the conserved momentum $\ZZ(u(t))=\int_{\R^{3}}2\IM\overline{u}(t)\nabla u(t)\,dx$ is zero. Moreover,
\begin{equation}\label{L:xt}
\left|\frac{x(t)}{t}\right|\rightarrow 0 \quad \text{as $t\rightarrow \infty$}.
\end{equation}
\end{lemma}
\begin{proof}
Assume that $\ZZ(u)\neq 0$.  We then define
\[
w(t,x)=e^{ix\cdot\xi_{0}}e^{-it|\xi_{0}|^{2}}u(t, x-2\xi_{0}t),
\]
where $\xi_{0}=-\ZZ(u)/(2M(u))$. As $E(P_{\omega})>0$ and 
\[
M(w)=M(u)=M(P_{\omega})<M(Q_{1}),
\]
the inequality \eqref{BoundH} implies that $E(w)>0$. Moreover.
\[
E(w)=E(u)-\frac{1}{2}\frac{|\ZZ(u)|^{8}}{M(u)}<E(u)=E(P_{\omega}).
\]
In particular, $(M(w), E(w))\in \K$.  Theorem~\ref{MainTheorem} implies that  $w$ scatters in $H^{1}(\R^{3})$, which is a contradiction because
\[
\| w\|_{L^{10}_{t,x}([0,\infty)\times\R^{3})}=\|  u \|_{L^{10}_{t,x}([0,\infty)\times\R^{3})}=\infty.
\]

Finally, the proof of \eqref{L:xt} is identical to the proof of Proposition 10.2 in \cite{KillipOhPoVi2017}.
\end{proof}

Next, we use the localized virial argument to show that $\delta(t_n)\to 0$ along some sequence $t_n\to\infty$.  This gives a preliminary indication of the convergence of $u$ to $P_\omega$.

\begin{lemma}\label{ZeroVirial}
Suppose $u$ is a solution of \eqref{NLS} satisfying the assumptions of Proposition~\ref{CompacDeca}. 
Then there exists a sequence $t_{n}\to\infty$ such that 
\[
\lim_{n\to \infty}\delta(t_{n})=0.
\]
\end{lemma}
\begin{proof} We use a localized virial argument.  First, recall $\delta(t)=V(u(t))$. By Lemma~\ref{VirialIden} we can write
\begin{equation}\label{DecVi}
\tfrac{d}{dt}\P_{R}(u(t))=8V(u(t))+[F_{R}(u(t))-F_{\infty}(u(t))].
\end{equation}
Moreover,
\begin{equation}\label{PC}
|\P_{R}(u)|\lesssim R\|u\|^{2}_{L^{\infty}_{t}H^{1}_{x}}\lesssim_{P_{\omega}} R
\end{equation}
and
\begin{equation}\label{OrdeP}
|F_{R}(u(t))-F_{\infty}(u(t))|=\mathcal{O}\biggl\{
\int_{|x|\geq R}[ |\nabla u|^{2} +|u|^{6} + |u|^{4}+|u|^{2}]\,dx\biggr\}.
\end{equation}
By compactness in $H^{1}(\R^{3})$ (see \eqref{CompactX}), for $\epsilon>0$ there exists $C(\epsilon)>0$ 
such that
\[
\sup_{t\in \R}\int_{|x-x(t)|>C(\epsilon)}[ |\nabla u|^{2} +|u|^{6} + |u|^{4}+|u|^{2}](t,x)dx\leq \epsilon.
\]
As $|x(t)|=o(t)$ as $t\to \infty$ (cf. Lemma~\ref{Mome}), it follows that there exists $T_{0}(\epsilon)$ such that
\[
|x(t)|\leq \epsilon t \quad \text{for all $t\geq T_{0}(\epsilon)$}.
\]
Given $T>T_{0}(\epsilon)$,  we put
\[
R:=C(\epsilon)+\sup_{t\in [T_{0}(\epsilon), T]}|x(t)|.
\]
As $\left\{x: |x|\geq R\right\}\subset \left\{x: |x-x(t)|\geq C(\epsilon)\right\}$ for $t\in [T_{0}(\epsilon), T]$,
 by \eqref{OrdeP} we get
\begin{equation}\label{EstimaP}
|F_{R}(u(t))-F_{\infty}(u(t))|\leq C\epsilon \quad \text{ for $t\in [T_{0}(\epsilon), T]$.}
\end{equation}

Combining this with \eqref{DecVi}, \eqref{PC} and integrating on $[0, T]$  we see that
\begin{align*}
\int^{T}_{0}V(u(t))\,dt&= \int^{T_{0}(\epsilon)}_{0}V(u(t))\,dt+	\int^{T}_{T_{0}(\epsilon)}V(u(t))\,dt\\
&\leq \int^{T_{0}(\epsilon)}_{0}V(u(t))\,dt+ \tilde{C}[\epsilon T + C(\epsilon)+\sup_{t\in [T_{0}(\epsilon), T]}|x(t)|].
\end{align*}
Since $\sup_{t\in [T_{0}(\epsilon), T]}|x(t)|\leq \epsilon T$, we derive 
\[
\tfrac{1}{T}\int^{T}_{0}V(u(t))\,dt\leq \tfrac{1}{T}\int^{T_{0}(\epsilon)}_{0}V(u(t))\,dt
+\tilde{C}\epsilon +\tfrac{\tilde{C} C(\epsilon)}{T}.
\]
Taking to the limit $T\to \infty$ and then $\epsilon\to 0$ we find
\[
\lim_{T\to \infty}\tfrac{1}{T}\int^{T}_{0}V(u(t))\,dt=0.
\]
In particular, there exists a sequence $t_{n}\to \infty$  and  $V(u(t_{n}))\to 0$ as $n\to \infty$.\end{proof}

Next, we use a more refined version of the localized virial estimate (taking the modulation analysis into account) to obtain a better estimate on $\delta(t)$. 

\begin{lemma}\label{Moduvirial}
Suppose $u$ is a solution of \eqref{NLS} satisfying the assumptions of Proposition~\ref{CompacDeca} with $x(t)$ defined in 
\eqref{Def:x}. Then there exists a positive constant 
$C$ such that for any interval $[t_{1}, t_{2}]\subset [0, \infty)$ we have
\begin{equation}\label{BoundT11}
\int^{t_{2}}_{t_{1}}\delta(t)dt\,\leq C\big[1+\sup_{t\in[t_{1},t_{2}]}|x(t)|\big]\left\{\delta(t_{1})+\delta(t_{2})\right\}.
\end{equation}
\end{lemma}
\begin{proof}
Let $\delta_{1}\in (0, \delta_{0})$ be sufficiently small (cf. Proposition~\ref{ModilationFree}).
We use the localized virial identity in Lemma~\ref{VirialModulate11}
with the function $\chi(t)$ satisfying
\[
\chi(t)=
\begin{cases}
1& \quad \delta(t)<\delta_{1}, \\
0& \quad \delta(t)\geq \delta_{1}.
\end{cases}
\]
Let $R>1$ to be specified later. Notice that we can write
\begin{equation}\label{Iden11}
\frac{d}{dt}\P_{R}[u(t)]=8\delta (t)+\sigma(t)
\end{equation}
where
\begin{equation}\label{Error11}
\sigma(t)=
\begin{cases}
F_{R}[u(t)]-F_{\infty}[u(t)]& \quad  \text{if $\delta(t)\geq \delta_{1}$}, \\
F[u(t)]-F_{\infty}[u(t)]-\Gamma[u(t)]& \quad \text{if $\delta(t)< \delta_{1}$},
\end{cases}
\end{equation}
with
\begin{equation}\label{Error22}
\Gamma(t)=F_{R}[e^{i \omega t}e^{i\theta(t)}P_{\omega}(\cdot-y(t))]-F_{\infty}[e^{i \omega t}e^{i\theta(t)}P_{\omega}(\cdot-y(t))].
\end{equation}
We will show the following:

\begin{enumerate}[label=\rm{(\roman*)}]
\item \textit{{Claim I.}} Fix $R>1$.  We have
\begin{align}\label{Ident22}
|\P_{R}[u(t_{j})]|\lesssim \tfrac{R}{\delta_{1}}\delta(t_{j}) \quad& \text{if $\delta(t_{j})\geq \delta_{1}$ for $j=1$, $2$},\\
\label{EstimateV22}
|\P_{R}[u(t_{j})]|\lesssim R \delta(t_{j}) \quad &\text{if $\delta(t_{j})< \delta_{1}$ for $j=1$, $2$}.
\end{align}
\item \textit{{Claim II.}} For  $\epsilon>0$, there exists $\rho_{\epsilon}>0$ so that if
$R=\rho_{\epsilon}+\sup_{t\in [t_{1}, t_{2}]}|x(t)|$, then 
\begin{align}\label{EstimateE11}
|\sigma(t)|\lesssim \epsilon \quad &\text{uniformly for $t\in [t_{1}, t_{2}]$ and $\delta(t)\geq \delta_{1}$},\\
\label{EstimateE22}
|\sigma(t)|\leq \epsilon \delta(t)\quad &\text{uniformly for $t\in [t_{1}, t_{2}]$ and $\delta(t)< \delta_{1}$}.
\end{align}
\end{enumerate}

By using \eqref{Ident22}, \eqref{EstimateV22}, \eqref{EstimateE11}, \eqref{EstimateE22} and applying the fundamental theorem 
of calculus to  \eqref{Iden11} over $[t_{1}, t_{2}]$ we get
\[
\int^{t_{2}}_{t_{1}}\delta(t)dt\lesssim 
\tfrac{1}{\delta_{1}}\Big[\rho_{\epsilon}+\sup_{t\in [t_{1}, t_{2}]}|x(t)|\Big](\delta(t_{1})+\delta(t_{2}))
+(\tfrac{\epsilon}{\delta_{1}}+\epsilon)\int^{t_{2}}_{t_{1}}\delta(t)dt,
\]
Thus, choosing $\epsilon=\epsilon(\delta_{1})>0$ sufficiently small
we get \eqref{BoundT11}.

It remains to establish the above claims.
\begin{proof}[{Proof of Claim I}] Assume that $\delta(t_{j})\geq \delta_{1}$. Notice that
\[
|\P_{R}[u(t)]|\lesssim R\|u\|^{2}_{L^{\infty}_{t}H^{1}_{x}}\lesssim_{P_{\omega}} \tfrac{R}{\delta_{1}}\delta(t_{j})
\]
which implies \eqref{Ident22}. Now, if $\delta(t_{j})< \delta_{1}$, then as $P_{\omega}$ is real valued, we write
\[
P_\omega(t) = e^{i[\theta(t)+\omega t]} P_\omega(\cdot-y(t))
\]
and obtain
\begin{align*}
|\P_{R}[u(t)]|&=\bigg|2\IM\int_{\R^{3}}\nabla w_{R}\cdot [\overline{u}(t_j)\nabla u(t_j)-\bar P_\omega(t_j)\nabla P_{\omega}(t_j)]dx\bigg|\\
&\lesssim
R[\|u\|_{L^{\infty}_{t}H^{1}_{x}}+\|P_{\omega}\|_{H^{1}}]
\|u(t_{j})-P_\omega(t_j)\|_{H^{1}}\\
&\lesssim_{P_{\omega}}R \delta(t_{j}).
\end{align*}
where in the last line we have used estimate \eqref{EstimateFree} in  Proposition~\ref{ModilationFree} 
\end{proof}

\begin{proof}[{Proof of Claim II}]
Suppose that $\delta(t)\geq \delta_{1}$. Using compactness (see \eqref{CompacNew}) we obtain that for each $\epsilon>0$, there 
exists a positive constant $\rho_{\epsilon}=\rho (\epsilon)>0$  such that
\begin{equation}\label{Compact123}
\int_{|x-x(t)|>\rho_{\epsilon}}|\nabla u(t,x)|^{2}+|u(t,x)|^{6}+|u(t,x)|^{6}+ |u(t,x)|^{2}dx<\epsilon.
\end{equation}
We put
\[
R:=\rho_{\epsilon}+\sup_{t\in [t_{1},t_{2}]}|x(t)|.
\]
Note that $\left\{|x|\geq R\right\}\subset \left\{|x-x(t)|\geq \rho_{\epsilon}\right\}$ 
for all $t\in [t_{1}, t_{2}]$. Then the same argument developed in the proof of Lemma~\ref{ZeroVirial} shows that (see \eqref{EstimaP})
\[
|F_{R}(u(t))-F_{\infty}(u(t))|\lesssim \epsilon.
\]
This proves \eqref{EstimateE11}.

Finally, we show  estimate \eqref{EstimateE22}.  Suppose  $\delta(t)< \delta_{1}$. 
In order to simplify the notation, we set $P_{\omega}(t):=e^{i \omega t}e^{i\theta(t)}P_{\omega}(\cdot-y(t))$.
By definition of $\sigma(t)$ when $\delta(t)<\delta_{1}$ (cf. \eqref{Error11}) we have  for $t\in [t_{1}, t_{2}]$, 
\begin{align*}
\sigma(t)&=\int_{|x|> R}8[|\nabla u|^{2}-|\nabla P_{\omega}(t)|^{2}]+
8[| u|^{6}-|P_{\omega}(t)|^{6}]\,dx \\
& \quad +\int_{|x|>R}-6[|u|^{4}-|P_{\omega}(t)|^{4}] +[-\Delta^{2}w_{R}][|u|^{2}-|P_{\omega}(t)|^{2}]\,dx\\
&\quad \int-\Delta w_{R}[|u|^{4}-|P_{\omega}(t)|^{4}]+4\RE[\overline{u_{j}}u_{k}-\partial_{j}\overline{P_{\omega}}(t)\partial_{k}P_{\omega}(t)]\partial_{j k}w_{R}(x)\,dx
\end{align*}
An application of the elementary inequality \begin{equation}\label{ElemIne}
||z_{1}|^{p+1}-|z_{2}|^{\alpha+1}|\lesssim|z_{1}-z_{2}|(|z_{1}|^{p}+|z_{2}|^{p})
\quad(p>0)
\end{equation}
immediately yields
\begin{equation*}
|\sigma(t)|
\lesssim \sum_{p\in\left\{1,3,5\right\}}
[\|u(t)\|^{p}_{H_{x}^{1}(|x|\geq R)}+\|P_{\omega}(t)\|^{p}_{H_{x}^{1}(|x|\geq R)}]
\|u(t)-P_{\omega}(t)\|_{H_{x}^{1}}.
\end{equation*}

Combining Proposition~\ref{ModilationFree} (note that $x(t)=y(t)$ because $\delta(t)<\delta_{1}$) and 
\eqref{Compact123}, and choosing $\rho_\eps$ larger if necessary, we deduce that
\begin{equation}\label{PrimerIne}
|\sigma(t)|
\lesssim \sum_{p\in\left\{1,3,5\right\}}
[\|u(t)\|^{p}_{H_{x}^{1}(|x-x(t)|\geq \rho_{\epsilon})}+\|P_{\omega}\|^{p}_{H_{x}^{1}(|x|\geq \rho_{\epsilon})}] \delta(t)
\lesssim_{P_{\omega}}
\epsilon \delta(t).
\end{equation}
\end{proof}
With both claims established, we complete the proof of the lemma. \end{proof}

The last ingredient we need before proving Proposition~\ref{CompacDeca} is the following lemma, which shows how $\delta(t)$ controls the motion of $x(t)$. 

\begin{lemma}\label{SpatialcenterCombi}
There exists $C>0$ such that
\begin{equation}\label{InequDelta}
|x(t_{1})-x(t_{2})|\leq C\int^{t_{2}}_{t_{1}}\delta(t)\,dt,
\end{equation}
for all $t_{1}$, $t_{2}>0$ with $t_{1}+1\leq t_{2}$.
\end{lemma}
\begin{proof}
We follow \cite[Lemma 6.8]{DuyckaertsRou2010} and divide the proof into three steps. 

\textsl{Step 1.} There exists a constant $C>0$ such that
\begin{equation}\label{step11}
|x(t)-x(s)|\leq C \quad \text{for all $t$, $s\geq 0$ such that $|t-s|\leq 2$}.
\end{equation}
The proof of \eqref{step11} is the same as the one given in  \cite[Lemma 3.10, Step 1]{DuyMerle2009}.

\textsl{Step 2.} Let $\delta_{0}>0$ be as Proposition~\ref{ModilationFree}.
We will show that there exists $\delta_{\ast}>0$ such that either
\begin{equation}\label{MinMax11}
\qtq{for all $T\geq 0$, either}\inf_{t\in [T, T+2]}\delta(t)\geq \delta_{\ast} \qtq{or}
\sup_{t\in [T, T+2]}\delta(t)<\delta_{0}.
\end{equation}
Assume instead that there exist $t_{n}^{\ast}\geq 0$ and two sequences
$t_{n}$, $t^{\prime}_{n}\in  [t_{n}^{\ast}, t_{n}^{\ast}+2]$ such that (up to a subsequence)
\begin{align}\label{ContraStep2}
&\delta(t_{n})\to 0 \quad \text{and}\quad \delta(t^{\prime}_{n})\geq \delta_{1} \quad \text{as $n\to \infty$},\\
\label{ContraLimit}
& t^{{\prime}}_{n}-t_{n}\to t^{\ast}\in[-2,2].
\end{align}
By \eqref{CompacNew}, we dedeuce that there exits $\varphi\in H^{1}(\R)$ such that 
\begin{equation}\label{Step2Conver}
\text{$u(t_{n}, \cdot+x(t_{n}))\to \varphi$ strongly in $H^{1}(\R^{3})$ as $n\to \infty$}.
\end{equation}
In particular, as $\delta(t_{n})\to 0$ and $(M(u), E(u))=(M(P_\omega),E(P_\omega))$, Lemma~\eqref{PriModula} implies that $\varphi(x)=e^{i \theta_{0}}P_{\omega}(x-x_{0})$ for some $\theta_{0}\in \R$ and  $x_{0}\in \R^{3}$.
Notice also that the solution to \eqref{NLS} with initial data $\varphi$ is $f(t,x)=e^{i(t+\theta_{0})}P_{\omega}(x-x_{0})$.
Thus, by the continuity of the flow and \eqref{ContraLimit} we infer that \[
\delta(t^{\prime}_{n})\to V(e^{i(t^{\ast}+\theta_{0})}P_{\omega}(x-x_{0}))=0,\]
which contradicts \eqref{ContraStep2}. This completes the proof of \eqref{step11}.

\textsl{Step 3.} Conclusion. We prove \eqref{InequDelta} with an additional condition that $t_{2}\leq t_{1}+2$. By \eqref{MinMax11} we have two cases:
\begin{enumerate}[label=\rm{(\roman*)}]

\item  If $\inf_{t\in [t_{1}, t_{2}]}\delta(t)\geq \delta_{\ast}$ holds, then \eqref{InequDelta} follows immediately 
 by applying \eqref{step11}.

\item On the other hand,  if $\sup_{t\in [t_{1}, t_{2}]}\delta(t)<\delta_{1}$ holds, then as $x(t)=y(t)$ for all $t\in I_{0}$,
from  \eqref{EstimateFree} we have $|x^{\prime}(t)|\leq C \delta(t)$ for all $t\in I_{0}$. Applying  the fundamental theorem of calculus
we get \eqref{InequDelta}.
\end{enumerate}

Finally we may remove the assumption $t_{2}\leq t_{1}+2$ by dividing the interval $[t_{1}, t_{2}]$ into intervals of length at least $1$ and
at most $2$ and combining together the inequalities in (i) and (ii).
\end{proof}

\begin{proof}[{Proof of Proposition~\ref{CompacDeca}}]
Combining Lemmas~\ref{ZeroVirial}, \ref{Moduvirial} and \ref{SpatialcenterCombi}, and using the same argument developed in  \cite[Proposition 6.1]{DuyckaertsRou2010}, we can show that $|x(t)|$ is bounded on $[0,\infty)$.  Briefly, we use the standard localized virial to find a sequence $t_n\to\infty$ with $\delta(t_n)\to 0$.  Using the modulated virial and the fact that the integral of $\delta$ controls the variation of $x(\cdot)$, we can obtain that $|x(t)|\lesssim |x(t_N)|$ for all $t\geq t_N$ for some sufficiently large $N$.

Applying Lemma~\ref{Moduvirial}, we then obtain that there exists $C>0$ so that
\[
\int^{s}_{T}\delta(t)\,dt\leq C\left\{\delta(T)+\delta(s)\right\}
\]
with $[T, s]\subset [0, \infty]$.  Applying this with a sequence $t_n\to\infty$ such that $\delta(t_n)\to 0$, we find that $\int^{\infty}_{T}\delta(t)\,dt\leq C\delta(T)$ for all $T\geq 0$.  Gronwall's lemma then implies that there exists $\alpha$, $\beta>0$ so such
\[
\int^{\infty}_{T}\delta(t)\,dt\leq \alpha e^{-\beta T}.
\]
The desired convergence then follows from Lemma~\ref{Decaiment}.
\end{proof}

\begin{proof}[{Proof of Corollary~\ref{ClassC}}]
Assume that $u$ satisfies \eqref{PropCon11} and \eqref{Infinity10}. Arguing as above, we can construct $x(t)$ such that $\left\{u(t+x(t)): t\in \R\right\}$ is  pre-compact in $H^{1}$.  Furthermore, we can prove that $x(t)$ is bounded and 
\[
\lim_{t\to -\infty}\delta(t)=\lim_{t\to \infty}\delta(t)=0.
\]
Modifying the proof of Lemma~\ref{Moduvirial}, one obtains
\[
\int^{n}_{-n}\delta(t)\,dt\leq C(\delta(n)+\delta(-n))\quad \text{for all $n\in \N$}.
\]
Sending $n\to \infty$, we obtain $V(u(t))=\delta(t)\equiv 0$, contradicting \eqref{PropCon11}.\end{proof}

\section{Construction of local stable solutions}\label{S: stableS}

In this section, we establish the existence and uniqueness of the solution converging exponentially to the soliton $P_{\omega}$.

We begin with the construction of some approximate solutions to the linearized equation \eqref{Decomh}.  The proof is similar to that of \cite[Proposition 6.3]{DuyMerle2009}, so it will suffice to sketch the argument. We recall the notation for the eigenfunctions and eigenvalues of $\mathcal{L}$ introduced in Section~\ref{S:Spectral} (see e.g. Lemma~\ref{SpecLL}).

\begin{proposition}\label{ApproxSo}
Let $a\in \R$. There exist $\{g^{a}_{j}\}_{j\geq 1}$ in $\Sch(\R^{3})$ such that the following holds: writing $g^{a}_{1}=ae_{+}$ and
\[
W^{a}_{k}(t,x):=\sum^{k}_{j=1}e^{-j\lambda_{1} t}g^{a}_{j}(x)\qtq{for}k\geq 1,
\]
we have 
\begin{equation}\label{Aproxlimi}
\partial_{t}W^{a}_{k}+\L W^{a}_{k}=iR(W^{a}_{k})+\mathcal{O}(e^{-(k+1)\lambda_{1}t})\quad \text{in $\Sch(\R^{3})$}\qtq{as}t\to\infty,
\end{equation}
where the nonlinear terms are defined in \eqref{Resi}.
\end{proposition}
\begin{proof} We prove this proposition by induction. To simplify notation, we omit most superscripts.

We define $g_{1}=ae_{+}$ and $W_{1}(t,x):=e^{-\lambda_{1}t}g_{1}(x)$. It is clear that
\[
\partial_{t}W_{1}+\L W_{1}-iR(W_{1})=-iR(W_{1})=-iR(e^{-\lambda_{1}t}g_{1}),
\]
Since $iR(e^{-\lambda_{1}t}g_{1})=O(e^{-2\lambda_{1}t})$ we obtain \eqref{Aproxlimi} for $k=1$.

Let $k\geq 1$. We assume that $g_{1}, g_{2}, \ldots, g_{k}$ and the corresponding $W_{k}$ satisfying \eqref{Aproxlimi} have been constructed.  Using the explicit expression for $R$, we can write
\[
R(W_{k})=\sum^{5k}_{j=2}e^{-j\lambda_{1}t}\psi_{jk} \quad \text{with $\psi_{jk}\in \Sch(\R^{3})$}.
\] 
Combining this with  \eqref{Aproxlimi} we deduce that as $t \to \infty$, we have
\begin{equation}\label{Aproxlimi11}
\partial_{t}W_{k}+\L W_{k}=iR(W_{k})+e^{-(k+1)\lambda_{1}t}V_{k+1}+\mathcal{O}(e^{-(k+2)\lambda_{1}t})\quad \text{in $\Sch(\R^{3})$}
\end{equation}
for some $V_{k+1}\in \Sch(\R^{3})$.  Noting that $(k + 1)\lambda_{1}$  is not in the spectrum of $\L$ (recall that $k\geq1$), we define 
\[
g_{k+1}:=-(\L-(k + 1)\lambda_{1})^{-1}V_{k+1}.
\]
As $V_{k+1}\in \Sch(\R^{3})$, it follows that $g_{k+1}\in \Sch(\R^{3})$ (see Remark 7.2 in \cite{DuyMerle2009}).
Now, we set 
\[
W_{k+1}(t,x):=W_{k}(t,x)+e^{-(k+1)\lambda_{1}t}g_{k+1}(x).
\]
By definition of $W_{k+1}$ and \eqref{Aproxlimi11}, we get
\begin{equation}\label{AproxlimiK1}
\begin{aligned}
\partial_{t}W_{k+1}&+\L W_{k+1}-iR(W_{k+1})\\
&=iR(W_{k})-iR(W_{k+1})+\mathcal{O}(e^{-(k+2)\lambda_{1}t})\quad \text{in $\Sch(\R^{3})$}.
\end{aligned}
\end{equation}
Since $W_{j}=O(e^{-\lambda_{1}t})$ for $j=1$, $2$, $\ldots$ $k$ and $W_{k}-W_{k+1}=O(e^{-(k+1)\lambda_{1}t})$ in $\Sch(\R^{3})$
as $t\to \infty$, if follows that $iR(W_{k})-iR(W_{k+1})=O(e^{-(k+2)\lambda_{1}t})$ in $\Sch(\R^{3})$
as $t\to \infty$ (cf. \eqref{ElemIne}). Therefore, using the expansion  \eqref{AproxlimiK1}  we have that \eqref{Aproxlimi} holds at $k+1$, which completes the induction. \end{proof}

We will use the approximate solutions $W_k^a(t)$ to construct a true solution to \eqref{Decomh}; however, we first need to collect a few technical lemmas.

The linearized equation \eqref{Decomh} may be written as a Schr\"odinger equation 
\begin{equation}\label{EquiVale}
i\partial_{t}h+\Delta h-\omega h +[\Lambda_{1}(h)+iR_{1}(h)]+[\Lambda_{2}(h)+iR_{2}(h)]=0,
\end{equation}
where 
\begin{align*}
\Lambda_{1}(h)&=2P^{2}_{\omega}h+P^{2}_{\omega}\overline{h},\\
\Lambda_{2}(h)&=-3P^{4}_{\omega}h-2P^{4}_{\omega}\overline{h},\\
R_{1}(h)&=|h|^{2}h+P_{\omega}[2|h|^{2}+h^{2}],\\
R_{2}(h)&=-h|h|^{4}+P_{\omega}[-2|h|^{2}h^{2}-3|h|^{4}]
-2P^{2}_{\omega}[\tfrac{h^{3}}{2}+3|h|^{2}h +\tfrac{3}{2}|h|^{2}\overline{h}]\\
&\quad -2P^{3}_{\omega}[2|h|^{2}+\tfrac{5}{2}h^{2}+\tfrac{1}{2}(\overline{h})^{2}].
\end{align*}

We first have the following nonlinear estimates:
\begin{lemma}\label{StriZ}
Let  $I$ be a finite interval of length $|I|$.  Then, there exists $\alpha>0$ and a constant $C$ independent of $I$ such that
\begin{equation}\label{R11Diference}
\begin{aligned}
\|R_{1}&(f)-R_{1}(g)\|_{L_{t}^{\frac{5}{3}}H^{1,\frac{30}{23}}_{x}} \\
&\leq C\|f-g\|_{L^{10}_{t,x}}
[
(\|f\|_{L_{t}^{\infty}L^{2}_{x}}+1)\| f \|_{L_{t}^{2}H^{1,6}_{x}}
+
(\|g\|_{L_{t}^{\infty}L^{2}_{x}}+1)\| g \|_{L_{t}^{2}H^{1,6}_{x}}]\\
&\quad +C\| f-g \|_{L_{t}^{2}H^{1,6}_{x}}
[\|f\|_{L^{10}_{t,x}}(\|f\|_{L_{t}^{\infty}L^{2}_{x}}+1)+
\|g\|_{L^{10}_{t,x}}(\|g\|_{L_{t}^{\infty}L^{2}_{x}}+1)]
\end{aligned}
\end{equation}
and
\begin{equation}\label{R22Diference}
\begin{aligned}
\|R_{2}&(f)-R_{2}(g)\|_{L_{t}^{\frac{10}{9}}H^{1, \frac{30}{17}}_{x}}\\
&\leq C(1+|I|^{\alpha})\|f-g\|_{L^{10}_{t,x}}[\|g\|_{L_{t}^{2}H^{1,6}_{x}}+\|f\|_{L_{t}^{2}H^{1,6}_{x}}]
\big(1+\sum^{3}_{j=1}\|f \|^{j}_{L^{10}_{t,x}}+\| g \|^{j}_{L^{10}_{t,x}}\big)\\
&\quad +
C(1+|I|^{\alpha})\|f-g\|_{L_{t}^{2}H^{1,6}_{x}}
\sum^{3}_{j=1}\bigl(\|f \|^{j}_{L^{10}_{t,x}}+\| g \|^{j}_{L^{10}_{t,x}}\bigr),
\end{aligned}
\end{equation}
where all spacetime norms are over $I\times \R^{3}$.
In particular, we have
\begin{equation}\label{L26E}
\begin{split}
&\|R_{1}(f)\|_{L_{t}^{\frac{5}{3}}H^{1,\frac{30}{23}}_{x}}
\leq C
\|f\|_{L_{t}^{2}H^{1, 6}_{x}}\|f \|_{L^{10}_{t, x}}[1+\|f \|_{L_{t}^{\infty}L^{2}_{x}}]\\
&\|R_{2}(f)\|_{L_{t}^{\frac{10}{9}}H^{1, \frac{30}{17}}_{x}}
\leq 
C(1+|I|^{\alpha})\|f\|_{L_{t}^{2}H^{1, 6}_{x}}
\sum^{4}_{j=1}\|f \|^{j}_{L^{10}_{t,x}}.
\end{split}
\end{equation}
Moreover,
\begin{equation}\label{Lam}
\begin{split}
\|\Lambda_{1}(f)\|_{L_{t}^{\frac{5}{3}}H^{1,\frac{30}{23}}_{x}}
&\leq C
|I|^{\alpha}\|f\|_{L_{t}^{2}H^{1, 6}_{x}},\\
\|\Lambda_{2}(f)\|_{L_{t}^{\frac{10}{9}}H^{1, \frac{30}{17}}_{x}}
&\leq C
|I|^{\alpha}\|f\|_{L_{t}^{2}H^{1, 6}_{x}}.
\end{split}
\end{equation}
\end{lemma}
\begin{proof}
Let $F(z)=|z|^{p}z$ for $p\geq1$. We have the following pointwise estimate:
\begin{align}\label{PointE11}
|F(f)-F(g)|&	\lesssim
|f-g|(|f|^{p}+|g|^{p}),\\ \label{PointE22}
|\nabla F(f)-\nabla F(g)|&	\lesssim
|f-g|(|f|^{p-1}+|g|^{p-1})(|\nabla f|+|\nabla g|)\\
& \quad +|\nabla f-\nabla g|(|f|^{p}+|g|^{p}).
\end{align}
Moreover, using H\"older, we have
\begin{align}\label{Holder11}
\|fgh\|_{L_{t}^{\frac{5}{3}}H^{1,\frac{30}{23}}_{x}}&
\leq \|f \|_{L^{10}_{t,x}}\|g\|_{L_{t}^{2}L^{6}_{x}}\|h\|_{L_{t}^{\infty}L^{2}_{x}},\\\label{Holder22}
\|fghuv\|_{L_{t}^{\frac{10}{9}}H^{1, \frac{30}{17}}_{x}}&
\leq\|f \|_{L^{10}_{t,x}}\|g \|_{L^{10}_{t,x}}\|h \|_{L^{10}_{t,x}}\|u \|_{L^{10}_{t,x}}\|v\|_{L_{t}^{2}L^{6}_{x}}.
\end{align}
Combining \eqref{PointE11}--\eqref{Holder22}, we obtain \eqref{R11Diference}--\eqref{Lam}.\end{proof}

Next, we record a useful integral summation argument from \cite{DuyMerle2009}.
\begin{lemma}\label{SumsE}
Let $a_{0}>0$, $t_{0}>0$, $p\in[1, \infty)$, $E$ a normed vector space, and $f\in L_{\text{loc}}^{p}((t_{0}, \infty); E)$. Suppose that there exist $\tau_{0}>0$ and $C_{0}>0$ with 
\[
\|f\|_{L^{p}(t, t+\tau_{0})}\leq C_{0}e^{-a_{0}t} \quad \text{for all $t\geq t_{0}$}.
\]
Then for all $t\geq t_{0}$,
\[
\|f\|_{L^{p}(t, \infty)}\leq \frac{C_{0}e^{-a_{0}t}}{1-e^{-a_{0}\tau_{0}}}.
\]
\end{lemma}

We now construct true solutions to \eqref{NLS} that are close to the soliton as $t\to\infty$.

\begin{proposition}\label{ContracA}
Let $a\in \R$. There exist $k_{0}>0$ and  $t_{k}\geq 0$ such that for any $k\geq k_{0}$, there exists a solution $W^{a}$ to \eqref{NLS} such that for $t\geq t_k$, we have
\begin{equation}\label{Uniq}
\|W^{a}(t) - U_k^a(t)\|_{H^1}+
\|W^{a}-U^{a}_{k}\|_{{L_{t}^{2}H^{1, 6}_{x}}\cap {L_{t}^{10}H^{1, \frac{30}{13}}_{x}}((t, \infty)\times \R^{3})}
\leq e^{-(k+\frac{1}{2})\lambda_{1}t},
\end{equation}
where 
\[
U^{a}_{k}(t):=e^{i\omega t}(P_{\omega}+W^{a}_{k}(t)).
\]
In addition, $W^{a}$ is the unique solution to equation \eqref{NLS} satisfying \eqref{Uniq} for large $t$.  Finally, $W^{a}$ is independent
of $k$ and satisfies for large $t$,
\begin{equation}\label{UniqVec}
\|W^{a}(t)-e^{i\omega t}P_{\omega}-a e^{i\omega t} e^{-\lambda_{1} t}e_{+}\|_{H^{1}}\leq e^{-\frac{3}{2}\lambda_{1}t}.
\end{equation}
\end{proposition}

\begin{remark}\label{Wa:Positive}
Let $a\neq 0$. Then $V(W^{a}(t))>0$ for all $t\in \R$. Indeed, by \eqref{UniqVec}, conservation of mass and energy, we first note that
\[
M(W^{a})=M(P_{\omega})\quad \text{and}\quad E(W^{a})=E(P_{\omega}).
\]
Now suppose that there exists $t_{0}$ such that $V(W^{a}(t_{0}))=0$.
As $M(W^{a}(t_{0}))=M(P_{\omega})$ and $E(W^{a}(t_{0}))=E(P_{\omega})$, the variational characterization from \cite{KillipOhPoVi2017} and uniqueness for \eqref{NLS} implies that $W^{a}(t)=e^{i \theta} e^{i \omega t}P_{\omega}(\cdot-y_{0})$ for some $\theta\in \R$ and $y_{0}\in \R^{3}$. In this case, \eqref{UniqVec} shows that $\theta=0$ and $y_{0}=0$, so that $W^{a}(t)=e^{i \omega t}P_{\omega}$.  Then \eqref{UniqVec} and the fact that $a\neq 0$ imply
\[
\|e_{+}\|_{H^{1}_{x}}\leq e^{-\frac{\lambda_{1}}{2}t},
\]
for large $t>0$, which is a contradiction. \end{remark}

\begin{proof} Define 
\[
\Lambda(h)=2P^{2}_{\omega}h+P^{2}_{\omega}\bar h-3P^{4}_{\omega}h-2P^{4}_{\omega}\bar h
\]
and recall the functions $W_k^a$ constructed in Proposition~\ref{ApproxSo}, which satisfy
\begin{equation}\label{ErroB}
\epsilon_{k}:=\partial_{t}W^{a}_{k}+\L W^{a}_{k}-i R(W^{a}_{k})=O(e^{-(k+1)\lambda_{1}t})\quad \text{in $\Sch(\R^{3})$}.
\end{equation}
We wish to construct a suitable solution to 
\begin{equation}\label{NewW}
i\partial_{t}v+\Delta v-\omega v=-\Lambda(W^{a}_{k}+v)+\Lambda(W^{a}_{k})+R(W^{a}_{k}+v)-R(W^{a}_{k})-i\epsilon_{k}.
\end{equation}
Indeed, this equation may equivalently be written as 
\[
\partial_{t}v+\L v=-iR(W^{a}_{k}+v)+iR(W^{a}_{k})-\epsilon_{k},
\]
from which we can deduce that $W_k^a+v$ solves \eqref{Decomh}. In particular the desired solution to \eqref{NLS} may be defined as
\[
W^a(t) = e^{i\omega t}P_\omega + e^{i\omega t}[W_k^a(t)+v(t)].
\]

We construct the solution to \eqref{NewW} via a fixed point argument.  We define the operator
\begin{equation}\label{PointFix}
[\Phi v](t):=-\int^{\infty}_{t}e^{i(t-s)\Delta}[-\Lambda(v(s))+R(W^{a}_{k}(s)+v(s))-R(W^{a}_{k}(s))+\epsilon_{k}(s)]\,ds
\end{equation}
and the spaces
\begin{align*}
X(t, \infty)&:={L_{t}^{2}H^{1, 6}_{x}}\cap {L_{t}^{10}H^{1, \frac{30}{13}}_{x}}((t, \infty)\times \R^{3}),\\
N(t, \infty)&:=L_{t}^{\frac{10}{9}}H^{1, \frac{30}{17}}_{x}((t,\infty)\times \R^{3})+L_{t}^{\frac{5}{3}}H^{1,\frac{30}{23}}_{x}((t,\infty)\times \R^{3}).
\end{align*}
We now fix $k$ and $t_k$.  We will show that the map $\Phi$ defined above is a contraction on the Banach space
\[
B^{k}:=\left\{v\in E^{k}, \|v\|_{E^{k}}\leq 1\right\},
\]
where
\begin{align*}
E^{k}&:=\big\{v\in C_{t}H_{x}^{1}(t_{k}, \infty) \cap X(t_{k}, \infty),\ 
\|v\|_{E^{k}}<\infty\big\}, \\
\|v\|_{E^{k}}&=\sup_{t\geq t_{k}}e^{(k+\frac{1}{2})\lambda_{1}t}[\|v(t)\|_{H_{x}^{1}}+\|v\|_{ X(t, \infty)}].
\end{align*}

By Strichartz, for $v,u\in B^{k}$, we have
\begin{equation}\label{Diffe11}
\begin{aligned}
\|\Phi & v(t)\|_{H_{x}^{1}}+\|\Phi v\|_{X(t, \infty)} \\
& \leq C^{\ast}[\|\Lambda(v)\|_{N(t, \infty)}+\|R(W^{a}_{k}+v)-R(W^{a}_{k})\|_{N(t, \infty)}+\|\epsilon_{k}\|_{N(t, \infty)}],
\end{aligned}
\end{equation}
\begin{equation}\label{Diffe22}
\begin{aligned}
\|&\Phi  v(t)-\Phi u(t)\|_{H_{x}^{1}}+\|\Phi v-\Phi u\|_{X(t, \infty)}\\
&\leq C^{\ast}[\|\Lambda(v-u)\|_{N(t, \infty)} + \|R(W^{a}_{k}+v)-R(W^{a}_{k}+u)\|_{N(t, \infty)}+\|\epsilon_{k}\|_{N(t, \infty)}].
\end{aligned}
\end{equation}
Here $C^{\ast}$ encodes the various constants appearing in the Strichartz estimates.

We now need the following: 
\begin{claim}\label{Est34}
Let $v$, $u \in E^{k}$. There exists $k_{0}>0$ such that for $k\geq k_{0}$, we have
\begin{align}\label{Cla11}
\|\Lambda (u-v)\|_{N(t, \infty)}&\leq \tfrac{1}{4 C^{\ast}}e^{-(k+\frac{1}{2})\lambda_{1}t}\|u-v\|_{E^{k}},\\ 
\label{Cla22}
\|R(W^{a}_{k}+v)-R(W^{a}_{k}+u)\|_{N(t, \infty)}&\leq C_{k}e^{-(k+\frac{3}{2})\lambda_{1}t}\|u-v\|_{E^{k}},\\
\label{Cla33}
\|\epsilon_{k}\|_{N(t, \infty)}&\leq C_{k}e^{-(k+1)\lambda_{1}t},
\end{align}
for all $t\geq t_{k}$, where the constant $C_{k}$ depends only on $k$.
\end{claim}
\begin{proof}[{Proof of Claim~\ref{Est34}}]
We first estimate \eqref{Cla11}. Let $\tau_{0}>0$. Using \eqref{Lam}, we have
\[
\|\Lambda (u-v)\|_{N(t, t+\tau_{0})}\leq C_{1}\tau_{0}^{\alpha}e^{-(k+\frac{1}{2})\lambda_{1}t}\|u-v\|_{E^{k}}.
\]
We then obtain \eqref{Cla11} from Lemma~\ref{SumsE} for $k\geq k_{0}$ by choosing $\tau_{0}$ and $k_{0}$ appropriately.

Now we show \eqref{Cla22}. Recall that by construction $\|W^{a}_{k}(t)\|_{H_{x}^{1} \cap X(t, \infty)}\leq C_{k}e^{-\lambda_{1}t}$ 
(cf. Proposition~\ref{ApproxSo}).
By \eqref{R11Diference} and \eqref{R22Diference} (and recalling that $I=[t,t+1])$, we deduce that for $v$, $u \in E^{k}$,
\begin{align*}
\|R(W^{a}_{k}+v)-R(W^{a}_{k}+u)\|_{N(t,t+1)}&\leq C_{k,1}e^{-\lambda_{1}t}\|u-v\|_{H_{x}^{1} \cap X(t, t+1)} \\
&
\leq  C_{k,2}e^{-( k+\frac{3}{2}) \lambda_{1}t}\|u-v\|_{E^{k}},
\end{align*}
where the constant $ C_{k,2}$ depends only on $k$. Thus, Lemma~\ref{SumsE} implies \eqref{Cla22}. 

Finally, the estimate \eqref{Cla33} is a direct consequence of \eqref{ErroB}.\end{proof}

Now let $k\geq k_{0}$, where $k_{0}$ is defined in Claim~\ref{Est34}. Combining \eqref{Diffe11},  
\eqref{Diffe22}, \eqref{Cla11}, \eqref{Cla22} and  \eqref{Cla33} we get for all $t\geq t_{k}$,
\begin{align}\label{BoundPhi}
&\|\Phi v\|_{E^{k}}\leq (\tfrac{1}{4}+C^{\ast} C_{k}e^{-\lambda_{1}t_{k}}+C^{\ast} C_{k}e^{-\frac{1}{2}\lambda_{1}t_{k}}),\\
\label{BoundPhi11}
&\|\Phi v-\Phi u\|_{E^{k}}\leq \|u-v\|_{E^{k}}
(\tfrac{1}{4}+C^{\ast} C_{k}e^{-\lambda_{1}t_{k}}),
\end{align}
so that $\Phi$ is a contraction provided $t_{k}$ is chosen sufficiently large. Therefore, 
for $k\geq k_{0}$, equation \eqref{NLS} has a unique solution $W^{a}$ satisfying the estimate \eqref{Uniq} for $t \geq t_{k}$.
Note also that all of the above still remains valid
for larger $t_{k}$; in particular,  the uniqueness still holds in the class of solution of \eqref{NLS} satisfying \eqref{Uniq} for
$t>t^{\prime}_{k}$ with $t^{\prime}_{k}$ is a real number larger than $t_{k}$. Moreover, by the uniqueness in the fixed point argument, one can also show that  $W^{a}$ does not depend on $k$ (cf. \cite[Proposition 6.3]{DuyMerle2009}). 

Finally, we prove \eqref{UniqVec}. By construction we have 
$\|\Phi v\|_{H_{x}^{1}}\leq Ce^{-( k+\frac{1}{2} )\lambda_{1}t}$. As $e^{i \omega t}v=W^{a}-U^{a}_{k}$ and $v=\Phi v$, we obtain
\[
\|W^{a}(t)-U^{a}_{k}(t)\|_{H_{x}^{1}}\leq Ce^{-( k+\frac{1}{2} )\lambda_{1}t}.
\]
This, together with the fact that $U^{a}_{k}=e^{i \omega t}P_{\omega}+a e^{i \omega t}e^{-\lambda_{1}t}e_{+}+O(e^{-2\lambda_{1}t})$
in $\Sch(\R^{3})$, yields \eqref{UniqVec}. \end{proof}

\section{A Uniqueness Result}\label{S:Rigid}

Our first main goal in this section is to establish the following:

\begin{proposition}\label{UniqueU} Let $(M(P_\omega),E(P_\omega))\in\partial\mathcal{K}_s$.  If $u$ is a solution to \eqref{NLS} satisfying
\begin{equation}\label{UniqCon}
E(u)=E(P_\omega),\ M(u)=M(P_\omega),\qtq{and}
\|u-e^{i \omega t}P_{\omega}\|_{H_{x}^{1}}\leq Ce^{-ct} 
\end{equation} 
for some $C$, $c>0$, then there exists unique $a \in \R$ such that $u=W^{a}$, where $W^{a}$ is the solution of \eqref{NLS} constructed
in Proposition~\ref{ContracA}.
\end{proposition}

We begin with a lemma.

\begin{lemma}\label{BoundGra11}
Let $v$ be a solution of \eqref{EquiVale} with 
\begin{equation}\label{ExpH}
\|v(t)\|_{H^{1}_{x}}\leq Ce^{-c_{0}t}
\end{equation}
for some constants $C>0$ and $c_{0}>0$. Then for any admissible pair $(q, r)$ we have for $t$ large
\begin{equation}\label{BoundCo}
\|v\|_{L^{10}_{t, x}([t,\infty)\times \R^{3})}
+\|v\|_{L_{t}^{q}H^{1, r}_{x}([t,\infty)\times \R^{3})}
\leq Ce^{-c_{0}t}.
\end{equation}
\end{lemma}

\begin{proof} Given $t>0$ and $\tau\in(0,1)$, let us write
\[
H(t)=\|v\|_{L_{t}^{2}H^{1, 6}_{x}\cap L_{t}^{10}H^{1, \frac{30}{13}}_{x}((t,t+\tau)\times \R^{3})}.
\]
By Strichartz and Lemma~\ref{StriZ}, we have
\begin{equation}\label{primerB}
\begin{split}
H(t)
\leq K\left\{ \|v(t)\|_{H^{1}_{x}}+[H(t)]^{2}+[H(t)]^{5}+\tau^{\alpha}H(t)
\right\}
\end{split}
\end{equation}
for some $\alpha>0$. Using \eqref{ExpH}, a continuity argument implies that there exists $\tau>0$ such that
\begin{equation}\label{NewBh}
\|v\|_{L_{t}^{2}H^{1, 6}_{x}\cap L_{t}^{10}H^{1, \frac{30}{13}}_{x}((t, t+\tau)\times \R^{3})}
\leq Ce^{-c_{0}t}
\end{equation}
for large $t$. Lemma~\ref{SumsE} and Sobolev embedding then yield
\[
\|v\|_{L^{10}_{t, x}([t,\infty)\times \R^{3})}+
\|v\|_{L_{t}^{2}H^{1, 6}_{x}\cap L_{t}^{10}H^{1, \frac{30}{13}}_{x}([t, \infty)\times \R^{3})}
\leq Ce^{-c_{0}t}.
\]
\end{proof}

Let $t_{0}\geq 0$ and suppose that we have functions
\[
v\in C^{0}([t_{0}, \infty), H^{1}(\R^{3})),\quad g_{1}, g_{2}\in C^{0}([t_{0}, \infty), L^{\frac{6}{5}}(\R^{3}))
\]
satisfying the following: 
\[
g_{1}\in L_{t}^{\frac{5}{3}}H^{1,\frac{30}{23}}_{x}([t_{0}, \infty)\times \R^{3}),\quad g_{2}\in L_{t}^{\frac{10}{9}}H^{1,\frac{30}{17}}_{x}([t_{0}, \infty)\times \R^{3}),
\]
and 
\begin{align}\label{CondiExp}
& \partial_{t}v+\L v=g_{1}+g_{2}, \quad (t,x)\in [t_{0}, \infty)\times \R^{3},\\ \label{CondiExp22}
& \|v(t)\|_{H^{1}}\leq Ce^{-c_{1}t},\\ \label{CondiExp33}
&\|g_{1}+g_{2}\|_{L_{x}^{\frac{6}{5}}(\R^{3})}+\|g_{1}\|_{L_{t}^{\frac{5}{3}}H^{1,\frac{30}{23}}_{x}([t, \infty)\times \R^{3})} +\|g_{2}\|_{L_{t}^{\frac{10}{9}}H^{1,\frac{30}{17}}_{x}([t, \infty)\times \R^{3})}\leq Ce^{-c_{2}t}
\end{align}
for all $t\geq t_{0}$, where $0 < c_{1} < c_{2}$.

Using  Strichartz estimates, \eqref{Lam}, and Lemma~\ref{SumsE}, we can obtain the following result.

\begin{lemma}\label{AxuST11}
Under the above assumptions \eqref{CondiExp}, \eqref{CondiExp22} and  \eqref{CondiExp33} with $0 < c_{1} < c_{2}$, we have
\begin{equation}\label{NewSt11}
\|v\|_{L_{t}^{q}H^{1, r}_{x}([t,\infty)\times \R^{3})}
\leq Ce^{-c_{1}t}
\end{equation}
for any admissible pair $(q, r)$. 
\end{lemma}

In what follows, we adopt the following notation: given $c>0$, we denote by $c^{-}$ a positive number arbitrary close to $c$ and
such that $0<c^{-}<c$.  Recall that $\lambda_1>0$ denotes the eigenvalue of the linearized operator, as introduced in Lemma~\ref{SpecLL}.

\begin{lemma}\label{AxuST}
Consider $v$, $g_{1}$ and $g_{2}$ satisfying  assumptions \eqref{CondiExp}, \eqref{CondiExp22}, \eqref{CondiExp33} with parameters $0<c_1<c_2$. Then for any admissible pair $(q, r)$, we have:
\begin{enumerate}[label=\rm{(\roman*)}]
\item If $\lambda_{1}\notin [c_{1}, c_{2})$, then 
\begin{equation}\label{BoundHsec}
 \|v(t)\|_{H^{1}}+ \|v\|_{L_{t}^{q}H^{1, r}_{x}([t,\infty)\times \R^{3})}  \leq Ce^{-c^{-}_{2}t}
\end{equation}
\item If $\lambda_{1}\in [c_{1}, c_{2})$, then there exists $a\in \R$ such that
\begin{equation}\label{BoundHsec22}
 \|v(t)-ae^{-\lambda_{1}t}e_{+}\|_{H^{1}}
+ \|v-ae^{-\lambda_{1}t}e_{+}\|_{L_{t}^{q}H^{1, r}_{x}([t,\infty)\times \R^{3})} 
\leq Ce^{-c^{-}_{2}t}.
\end{equation}
\end{enumerate}
\end{lemma}
\begin{proof}
We closely follow the argument in \cite[Proposition 5.9]{DuyMerle2009} and  \cite[Lemma 7.2]{DuyckaertsRou2010}.
Let $Y^{\bot}$ be the set of $h\in H^{1}(\R^{3})$ satisfying the orthogonality relations \eqref{NewCon} and \eqref{Newcon22}.
We decompose $v$ as
\begin{equation}\label{DecompV}
v(t)=\alpha_{+}(t)e_{+}+\alpha_{-}(t)e_{-}+\sum^{3}_{j=0}\beta_{j}(t)P_{\omega, j}+v^{\bot}(t),
\end{equation}
where $v^{\bot}(t)\in Y^{\bot}$ and
\[
P_{\omega, 0}=\frac{i P_{\omega}}{\|P_{\omega}\|_{L^{2}}}, \quad
P_{\omega, j}=\frac{\partial_{j}P_{\omega}}{\|\partial_{j} P_{\omega}\|_{L^{2}}}
\quad \text{for $j=1,2,3$}. 
\]
By Remark~\ref{PLf11}, we can normalize the eigenfunctions $e_{\pm}$ so that that $\F(e_{+}, e_{-})=1$ (recall the notation $\F$ from \eqref{Quadratic}).  Moreover, from Remark~\ref{PLf} and definition of $Y^{\bot}$ we have that
\begin{align*}
&	\alpha_{+}(t)=\F(v(t), e_{-}), \quad	\alpha_{-}(t)=\F(v(t), e_{+}),\\
& \beta_{j}(t)=(v(t)-\alpha_{+}(t)e_{+}-\alpha_{-}(t)e_{-}, P_{\omega, j})_{L^{2}}\quad \text{for $j=0,1,2,3$}.
\end{align*}
\textbf{Step 1. Decay estimates.} By condition \eqref{CondiExp33} and following the same argument developed in \cite[Lemma 7.2, Step 1]{DuyckaertsRou2010}, one can show that
\begin{align}\label{Exp11alfa00}
&|\alpha^{\prime}_{-}(t)-\lambda_{1}\alpha_{-}(t)|\leq C e^{-c_{2}t},	\quad |\alpha^{\prime}_{+}(t)+\lambda_{1}\alpha_{+}(t)|\leq C e^{-c_{2}t},\\ \label{Exp11alfa}
&	|\beta^{\prime}_{j}(t)|\leq C( \|v^{\bot}(t)\|_{L_{x}^{2}}+ e^{-c_{2}t}),
\end{align}
and 
\begin{align}\label{OnExp11}
& |\alpha_{+}(t)|\leq e^{-c^{-}_{2}t} \quad \text{if $c_{2}\leq \lambda_{1}$}\\  \label{OnExp22}
& |\alpha_{+}(t)-ae^{-\lambda_{1}t}|\leq e^{-c_{2}t} \quad \text{if $c_{2}> \lambda_{1}$},
\end{align}
where
\begin{align}\label{L:alfa}
a:=\lim_{t\to \infty}e^{\lambda_{1}t}\alpha_{+}(t).
\end{align}

Next, we show that
\begin{equation}\label{Fexp1}
\F(v(t))\leq Ce^{-(c_{1}+c_{2})t}.
\end{equation}
Indeed, notice that, since $\F(\L v, v)=0$ (cf. Remark~\ref{PLf}) we have
\begin{equation}\label{DecomF}
\frac{d}{dt}\F(v(t))=2\F(\partial_{t} v(t), v(t))=-2\F(\L v,v)+2\F(g_{1}+g_{2}, v)=2\F(g_{1}+g_{2}, v).
\end{equation}
Moreover, by the definition of the quadratic form $\F$, we have for any time-interval $|I|$ with $|I|<\infty$,
\begin{equation}\label{L1BoundF}
\begin{aligned}
\int_{I}|\F(g_{1}(t)+g_{2}(t), v(t))|dt &\leq
\| g_{1} \|_{L_{t}^{\frac{5}{3}}H^{1,\frac{30}{23}}_{x}}
\| v  \|_{L_{t}^{\frac{5}{2}}H^{1, \frac{30}{7}}_{x}}
+
\| g_{2} \|_{L_{t}^{\frac{10}{9}}H^{1, \frac{30}{17}}_{x}}
\| v  \|_{L_{t}^{10}H^{1, \frac{30}{13}}_{x}}\\
&\quad +
|I|\|g_{1}+g_{2}\|_{L_{t}^{\infty}L^{\frac{6}{5}}_{x}}\|v\|_{L_{t}^{\infty}L^{6}_{x}},
\end{aligned}
\end{equation}
where all spacetime norms are over $I \times \R^{3}$. Thus, by conditions \eqref{CondiExp22}-\eqref{CondiExp33} and Lemma~\ref{AxuST11}, we get
\[
\int^{t+1}_{t}|\F(g_{1}(t)+g_{2}(t), v(t))|dt\leq Ce^{-(c_{1}+c_{2})t}.
\]
In this case, Lemma~\ref{SumsE} implies that
\[
\int^{\infty}_{t}|\F(g_{1}(t)+g_{2}(t), v(t))|dt\leq Ce^{-(c_{1}+c_{2})t}.
\]
As $|\F(v(t))|\lesssim \|v(t)\|^{2}_{H^{1}}\to 0$ as $t\to \infty$ (cf. \eqref{CondiExp22}), from \eqref{DecomF} and inequality above we have
\[
|\F(v(t))|\leq \int^{\infty}_{t}|\F(g_{1}(t)+g_{2}(t), v(t))|dt\leq Ce^{-(c_{1}+c_{2})t}.
\]
This proves \eqref{Fexp1}.

\textbf{Step 2. Proof in the case when either  $\lambda_1\geq c_2$, or $\lambda_1<c_2$ and $a=0$.}
Combining \eqref{Exp11alfa00}, \eqref{OnExp11}  and \eqref{OnExp22} one can obtain (see \cite[Lemma 7.2, Step 3]{DuyckaertsRou2010})
\begin{align}\label{Decaalfa11}
|\alpha_{+}(t)|+|\alpha^{\prime}_{+}(t)|&\leq Ce^{-c^{-}_{2}t}\\
\label{Decaalfa22}
|\alpha_{-}(t)|+|\alpha^{\prime}_{-}(t)|&\leq Ce^{-c_{2}t}.
\end{align}
On the other hand, since $\F(e_{+}, v^{\bot})=\F(e_{-}, v^{\bot})=0$ (recall $v^{\bot}\in Y^{\bot}$),
$\F(e_{+}, e_{-})=1$, and  $\F(e_{+})=\F(e_{-})=0$ we obtain
\[
\F(v)=\F(v^{\bot})+2\alpha_{+}\alpha_{-}.
\]
Thus, by \eqref{Decaalfa11}, \eqref{Decaalfa22} and \eqref{Fexp1}, Proposition~\ref{NewCoer} implies that
\begin{equation}\label{EstimaV}
\|v^{\bot}(t)\|_{H^{1}}\lesssim \sqrt{|\F(v^{\bot})|}\leq Ce^{-\frac{(c_{1}+c_{2})}{2}t}.
\end{equation}
Using  \eqref{Exp11alfa} we can therefore obtain the following estimate of $\beta_{j}(t)$,
\begin{equation}\label{BetaEst}
|\beta(t)|\leq Ce^{-\frac{(c_{1}+c_{2})}{2}t}, \quad \text{for $j=0,1,2,3$}.
\end{equation}

From the decomposition \eqref{DecompV} and summing up estimates \eqref{Decaalfa11}, \eqref{Decaalfa22}, \eqref{EstimaV} and \eqref{BetaEst}, we arrive at 
\[
\|v(t)\|_{H^{1}}\leq Ce^{-\frac{(c_{1}+c_{2})}{2}t}.
\]
Therefore, $v$, $g_{1}$ and $g_{2}$ satisfies \eqref{CondiExp}, \eqref{CondiExp22} and \eqref{CondiExp33} 
with $c_{1}$ replaced by $c^{\ast}_{1}=\frac{(c_{1}+c_{2})}{2}$. An iteration argument now yields
\begin{equation}\label{ArguI}
\|v(t)\|_{H^{1}}\leq Ce^{-c^{-}_{2}t}.
\end{equation}
In this case, using Lemma~\ref{AxuST11}, we obtain the estimate \eqref{BoundHsec}-\eqref{BoundHsec22}.

\textbf{Step 3. Proof in the remaining cases.} If $\lambda_1<c_1$, then we have $\lambda_1<c_2$ and $a=0$, and hence we obtain the estimate in (ii) with $a=0$ using Step 2.  Thus, it suffices to consider the case $c_1\leq \lambda_1<c_2$ and $a\neq 0$. 

We set
\[
\nu(t)=v(t)-ae^{-\lambda_{1}t}e_{+}.
\]
Note that 
\[
\partial_{t}\nu(t)+\L \nu(t)=g_{1}(t)+g_{2}(t), \quad \|\nu(t)\|_{H^{1}}\leq Ce^{-c_{1}t}.
\]
Let $\tilde{\alpha}_{+}(t)=\F(\nu(t), e_{-})$. From \eqref{OnExp22} we infer that
\[
|e^{\lambda_{1}t}\tilde{\alpha}_{+}(t)|\leq Ce^{-(c_{2}-\lambda_{1})t}\to 0 \quad \text{as $t\to \infty$}.
\]
Therefore,  $\tilde{\alpha}_{+}(t)$, $g_{1}$ and $g_{2}$ satisfy all the assumptions of Step 2 (cf. \eqref{L:alfa}), hence, we have
\[
\|v(t)-ae^{-\lambda_{1}t}e_{+}\|_{H^{1}}=\|\nu(t) \|_{H^{1}} \leq Ce^{-c^{-}_{2}t},
\]
and Lemma~\ref{AxuST11} implies the estimate \eqref{BoundHsec22}. \end{proof}

We will also need the following lemma, whose proof is standard. 

\begin{lemma}\label{EstR}
 Let $f$, $g$ functions in $H^{1}(\R^{3})$. Define $R(\cdot)$ as in \eqref{Resi}.  Then
\begin{equation}\label{EstimatesR}
\begin{split}
\|R(f)-R(g)\|_{L^{\frac{6}{5}}}\leq C
\|f-g\|_{L^{6}}(\|f\|_{H^{1}}  + \|g\|_{H^{1}} + \|f\|^{4}_{H^{1}}  + \|g\|^{4}_{H^{1}}).
\end{split}
\end{equation}
\end{lemma}

We turn now to the proof of Proposition~\ref{UniqueU}.
\begin{proof}[Proof of Proposition~\ref{UniqueU}]
Let $u=e^{i \omega t}(P_{\omega}+h)$ be a solution of  \eqref{NLS} satisfying \eqref{UniqCon}.
Note that $h$ satisfies the linearized equation
\[
\partial_{t}h+\L h=iR_{1}(h)+iR_{2}(h).
\]
In particular, $h$ satisfies the equation \eqref{EquiVale}.

\textbf{Step 1.} We will show that there exists $a\in \R$ such that
\begin{equation}\label{Step1Est}
 \|h(t)-ae^{-\lambda_{1}t}e_{+}\|_{H^{1}}
+ \|h(t)-ae^{-\lambda_{1}t}e_{+}\|_{L_{t}^{q}H^{1, r}_{x}([t,\infty)\times \R^{3})} 
\leq Ce^{-2^{-}\lambda_{1}t}.
\end{equation}
First, we show that
\begin{align}\label{Dsc11}
	&\|h(t)\|_{H^{1}}\leq Ce^{-\lambda_{1}t}
\end{align}
and 
\begin{equation}\label{Dsc33}
\begin{split}
\|R_{1}(h)+R_{2}(h)\|_{L_{x}^{\frac{6}{5}}}+\|R_{1}(h)\|_{L_{t}^{\frac{5}{3}}H^{1,\frac{30}{23}}_{x}([t, \infty)\times \R^{3})}
\\	+\|R_{2}(h)\|_{L_{t}^{\frac{10}{9}}H^{1,\frac{30}{17}}_{x}([t, \infty)\times \R^{3})}\leq
	Ce^{-2\lambda_{1}t}.
\end{split}
\end{equation}
Indeed, by \eqref{UniqCon} and Lemma~\ref{BoundGra11} we obtain for any admissible pair $(q,r)$,
\begin{equation}\label{BoundCoNew11}
\|h\|_{L^{10}_{t, x}([t,\infty)\times \R^{3})}
+\|h\|_{L_{t}^{q}H^{1, r}_{x}([t,\infty)\times \R^{3})}
\leq Ce^{-ct}.
\end{equation}
Thus, by Lemmas~\ref{EstR} and \ref{StriZ}  we see that
\begin{align}\label{g1g2}
&\|R_{1}(h)+R_{2}(h)\|_{L_{x}^{\frac{6}{5}}(\R^{3})}\leq Ce^{-2ct}\\ \label{g1g2H}
&\|R_{1}(h)\|_{L_{t}^{\frac{5}{3}}H^{1,\frac{30}{23}}_{x}((t, t+1)\times \R^{3})}
	+\|R_{2}(h)\|_{L_{t}^{\frac{9}{10}}H^{1,\frac{30}{17}}_{x}((t, t+1)\times \R^{3})}\leq
	Ce^{-2ct}.
\end{align}
Moreover, by using Lemma~\ref{SumsE} and \eqref{g1g2H} we obtain
\begin{equation}\label{g1g2HNew}
\|R_{1}(h)\|_{L_{t}^{\frac{5}{3}}H^{1,\frac{30}{23}}_{x}((t, \infty)\times \R^{3})}
	+\|R_{2}(h)\|_{L_{t}^{\frac{9}{10}}H^{1,\frac{30}{17}}_{x}((t, \infty)\times \R^{3})}\leq
	Ce^{-2ct}.
\end{equation}
Thus $h$ satisfies the conditions in Lemma~\ref{AxuST} with $g_{1}=iR_{1}(h)$, $g_{2}=iR_{2}(h)$, $c_{1}=c$, and $c_{2}=2c$.
Therefore, we obtain
\[
\|h(t)\|_{H^{1}}\leq C(e^{-\lambda_{1}t}+e^{-\frac{3}{2}c t})
\]
If $\frac{3}{2}c>\lambda_{1}$, it is clear that \eqref{Dsc11} holds. If not, by using \eqref{UniqCon} (with $\frac{3}{2}c$ instead $c$), 
Lemma~\ref{AxuST} and an iteration argument gives the estimate \eqref{Dsc11}. Moreover, combining Lemma~\ref{BoundGra11}, estimates \eqref{L26E} and \eqref{EstimatesR}, and Lemma~\ref{SumsE} we get \eqref{Dsc33}.

Thus, we can apply Lemma~\ref{AxuST} to obtain \eqref{Step1Est}.

\textbf{Step 2.} We will use the induction argument to show that there exists $t_{0}\geq 0$ such that for all $m>0$,
\begin{equation}\label{BoundH1}
\|u(t)-W^{a}(t)\|_{H_{x}^{1}}
+\|u-W^{a}\|_{{L_{t}^{2}H^{1, 6}_{x}}\cap {L_{t}^{10}H^{1, \frac{30}{13}}_{x}}((t, \infty)\times \R^{3})}
\leq e^{-m t}\quad \text{for all $t\geq t_{0}$},
\end{equation}
where $W^{a}=e^{i \omega t}(P_{\omega}+h^{a})$ ( recall that $W^{a}$ is the solution to \eqref{NLS} defined in 
Proposition~\ref{ContracA}).  Indeed, combining \eqref{Step1Est} and estimate \eqref{Uniq}, we have that \eqref{BoundH1}  holds with $m=\frac{3}{2}\lambda_{1}$.
Now we show that if \eqref{BoundH1} holds for some $m=m_{1}>\lambda_{1}$, then it also holds for $m=m_{1}+\frac{\lambda_{1}}{2}$.
Indeed, recalling that $u=e^{i \omega t}(P_{\omega}+h)$ and $W^{a}=e^{i \omega t}(P_{\omega}+h^{a})$), we note that $h-h^{a}$ satisfies the equation 
\[
\partial_{t}(h-h^{a})+\L (h-h^{a})=iR(h)-iR(h^{a}).
\]
Since by hypothesis we have
\[
\|h(t)-h^{a}(t)\|_{H_{x}^{1}}
+\|h-h^{a}\|_{{L_{t}^{2}H^{1, 6}_{x}}\cap {L_{t}^{10}H^{1, \frac{30}{13}}_{x}}((t, \infty)\times \R^{3})}
\leq e^{-m_{1} t},
\]
it follows from \eqref{Dsc11} and \eqref{UniqVec}, and  Lemma~\ref{EstR}, 
\[
\|R(h)-R(h^{a})\|_{L_{x}^{\frac{6}{5}}}
\leq 	Ce^{-(\lambda_{1}+m_{1})t}
\]
Similarly, combining Lemmas~\ref{StriZ}, \ref{SumsE} and \ref{BoundGra11} we get
\[\begin{split}
&\|R_{1}(h)-R_{1}(h^{a})\|_{L_{t}^{\frac{5}{3}}H^{1,\frac{30}{23}}_{x}([t, \infty)\times \R^{3})}
+\|R_{1}(h)-R_{1}(h^{a})\|_{L_{t}^{\frac{9}{10}}H^{1,\frac{30}{17}}_{x}([t, \infty)\times \R^{3})}
\\
&\leq Ce^{-(\lambda_{1}+m_{1})t}.
\end{split}\]
Thus, from Lemma~\ref{AxuST} we obtain
\[
\|h(t)-h^{a}(t)\|_{H_{x}^{1}}
+\|h-h^{a}\|_{{L_{t}^{2}H^{1, 6}_{x}}\cap {L_{t}^{10}H^{1, \frac{30}{13}}_{x}}((t, \infty)\times \R^{3})}
\leq e^{-(m_{1}+\frac{3}{4}\lambda_{1} )t},
\]
which implies \eqref{BoundH1} with $m=m_{1}+\frac{\lambda_{1}}{2}$. Thus, estimate \eqref{BoundH1} follows by iteration for all $m>0$.

Finally, we are ready to finish the proof of the proposition. Combining \eqref{BoundH1} with $m=(k_{0}+1)\lambda_{1}$, where $k_{0}$ is defined in Proposition~\ref{ContracA}, and \eqref{Uniq} we find that
\[
\|u(t)-U_{k_{0}}^{a}(t)\|_{H_{x}^{1}}
+\|u-U_{k_{0}}^{a}\|_{{L_{t}^{2}H^{1, 6}_{x}}\cap {L_{t}^{10}H^{1, \frac{30}{13}}_{x}}((t, \infty)\times \R^{3})}
\leq e^{-(k_{0}+\frac{1}{2})\lambda_{1} t}.
\]
for large $t$.  By the uniqueness in Proposition~\ref{ContracA}, we finally obtain $u=W^{a}$. \end{proof}

Our next goal is to show that if $u$ is a solution as in Proposition~\ref{UniqueU} with positive virial, then $u=W^a$ for some $a\geq 0$.  We begin with the following lemma: 

\begin{lemma}\label{GMP} Let $(M(P_\omega),E(P_\omega))\in\partial\mathcal{K}_s$. Suppose $u_{0}\in H^{1}(\R^{3})$, with
\begin{equation}\label{CondiEMV}
E(u_{0})=E(P_{\omega}), \quad
M(u_{0})=M(P_{\omega}), \quad \text{and} \quad
V(u_{0})>0.
\end{equation}
Then we have 
\[
\|\nabla u(t)\|^{2}_{L^{2}}<\|\nabla P_{\omega}\|^{2}_{L^{2}}
\quad
\text{for all $t\in \R$},
\]
where $u(t)$ is the corresponding solution to \eqref{NLS} with initial data $u_{0}$.
\end{lemma}
\begin{proof}
Using Corollary~\ref{ClassC}, we can assume that
\[ 
\| u\|_{L^{10}_{t, x}((-\infty, 0)\times \R^{3})}<\infty,
\]
i.e., the solution $u(t)$ scatters for negative time. Therefore, there exists $\psi_{-}\in H^{1}(\R^{3})$ such that
\begin{equation}\label{ScN}
\lim_{t\to -\infty}\|u(t)-e^{it \Delta}\psi_{-}\|_{H^{1}}=0.
\end{equation}
Note also that
\[
E(u(t))+\tfrac{1}{2}\omega M(u(t))=E(P_{\omega})+\tfrac{1}{2}\omega M(P_{\omega})=\tfrac{1}{3}\|\nabla P_{\omega}\|^{2}_{L^{2}}
\quad
\text{for all $t\in \R$}.
\]
Since $u(t)-e^{it \Delta}\psi_{-}\to 0$ in $H^{1}$ and $e^{it \Delta}\psi_{-}\to 0$ in $L^{4}$ as
$t \to -\infty$, we see that
\begin{align*}
&\lim_{t\to -\infty}\|\nabla u(t)\|^{2}_{L^{2}}=\|\nabla \psi_{-}\|^{2}_{L^{2}},\\
&\tfrac{1}{2}\|\nabla \psi_{-}\|^{2}_{L^{2}}+\tfrac{1}{2}\omega M(P_{\omega})\leq \tfrac{1}{3}\|\nabla P_{\omega}\|^{2}_{L^{2}},
\end{align*}
which implies that there exists $t_{0}\in (-\infty, 0)$ such that
\[
\|\nabla u(t_{0})\|^{2}_{L^{2}}\leq \tfrac{2}{3}\|\nabla P_{\omega}\|^{2}_{L^{2}}<\|\nabla P_{\omega}\|^{2}_{L^{2}}
\]
(recall $\omega M(P_{\omega})>0$). Therefore, by Remark~\ref{ReEMG} we get $\|\nabla u(t)\|^{2}_{L^{2}}<\|\nabla P_{\omega}\|^{2}_{L^{2}}$ for all $t\in \R$.
\end{proof}

\begin{proposition}\label{coroA}
If $u$ satisfies the assumption \eqref{UniqCon} and $V(u_{0})>0$, then $u=W^{a}$ for some unique $a\geq 0$.
\end{proposition}
\begin{proof}
By Proposition~\ref{UniqueU}, it is enough to show that there is no solution $W^{a}$ to  \eqref{NLS} satisfying \eqref{UniqVec} with $a<0$. Indeed, suppose $W^{a}$ were such a solution. 
Using \eqref{UniqVec}, a direct calculation leads to
\begin{align*}
\|\nabla W^{a}(t)\|^{2}_{L^{2}}&=\|\nabla P_{\omega}\|^{2}_{L^{2}}+ae^{-\lambda_{1}t}\int_{\R^{3}}\nabla P_{\omega}e_{1}dx
	+O(e^{-\frac{3}{2}\lambda_{1}t}).
\end{align*}
From Remark~\ref{PE1}, we get $\int_{\R^{3}}\nabla P_{\omega}e_{1}dx\neq 0$. Replacing $e_{+}$ by $-e_{+}$ if necessary, we can assume that 
\[
\int_{\R^{3}}\nabla P_{\omega}e_{1}dx<0.
\]
Thus, if $a<0$, then $\|\nabla W^{a}(t)\|^{2}_{L^{2}}>\|\nabla P_{\omega}\|^{2}_{L^{2}}$ for large $t$, which is a contradiction 
to Lemma~\ref{GMP}. Therefore $a\geq 0$.
\end{proof}

%

As a corollary of Propositions~\ref{UniqueU} and \ref{coroA}, we see that (modulo time translation and rotation) all of the functions $W^a$ with $a>0$ are the same.

\begin{corollary}\label{coroCla}
Let $a>0$. Then there exists $\theta_{a}$,  $T_{a}\in \R$ such that 
\[W^{a}=e^{ i \theta_{a}}W^{+1}(t-T_{a}).\]
\end{corollary}
\begin{proof}
Let $a>0$ and choose $T_{a}$ such that $ae^{-\lambda_{1}T_{a}}=1$. Then by estimate \eqref{UniqVec} we get
\begin{equation}\label{W1}
\|e^{-i \omega T_{a}}W^{a}(t+T_{a})-e^{i\omega t}P_{\omega}-e^{(i\omega-\lambda_{1}) t}e_{+}\|_{H^{1}}
\leq e^{-\frac{3}{2}\lambda_{1}t}.
\end{equation}
Note that $e^{-i \omega T_{a}}W^{a}(t+T_{a})$ satisfies the assumption of Proposition~\ref{UniqueU}. Thus, there exists $\tilde{a}\geq 0$
such that 
\[e^{-i \omega T_{a}}W^{a}(t+T_{a})=W^{\tilde{a}}.\]
From \eqref{W1} and Proposition~\ref{ContracA} we infer that $\tilde{a}=1$. Therefore, \[
W^{a}=e^{ i \omega T_{a}}W^{+1}(t-T_{a}).
\]
\end{proof}

\section{Proofs of the main results}\label{S:proof} 

Finally, we are in a position to establish the main results, namely, Theorem~\ref{TH1} and Theorem~\ref{ScatSoliton}.

\begin{proof}[{Proof of Theorem~\ref{TH1}}] We let $(m,e)=(M(P_\omega),E(P_\omega))\in\partial\mathcal{K}_s$. We set 
\[
\G_{\omega}(t):=W^{+1}(t),
\]
where $W^{+1}$ is the global solution to \eqref{NLS} defined in Proposition~\ref{ContracA}.  
By Remark~\ref{Wa:Positive} we see that $V(\G_{\omega}(t))>0$ for all $t\in \R$.
Moreover, by Corollary~\ref{ClassC} we obtain that the solution $\G_{\omega}(t)$ scatters for negative time. 
If not, applying the above arguments to the solution $\overline{\G}_{\omega}(x, -t)$ we get the desired result.\end{proof}

\begin{proof}[{Proof of Theorem~\ref{ScatSoliton}}]
Let $u$ be a solution to \eqref{NLS} such that \[
(M(u), E(u))=(m,e)\in \partial\K_{s}
\]
and $V(u(0))>0$.  

Suppose that $u$ does not scatter, i.e.  $\|u\|_{L^{10}_{t, x}(\R\times \R^{3})}=\infty$. Then, replacing if necessary $u(t)$ by $\overline{u(-t)}$, Proposition~\ref{CompacDeca} implies that there exist $\theta\in \R$, $y_{0}\in \R^{3}$, $c$ and $C>0$ such that
\[
\|e^{-i\theta}u(t, \cdot+y_{0})-e^{i\omega t}P_{\omega}\|_{H^{1}}\leq Ce^{-c t}.
\]
Now, Propositions~\ref{UniqueU} and \ref{coroA} imply that $e^{-i\theta}u(\cdot, \cdot+y_{0})=W^{a}$ for some $a\geq 0$. As $V(W^{a}(t))>0$ for all $t\in \R$ (recall that $V(u(0))>0$),
it follows that $a>0$. But then, by Corollary~\ref{coroCla} we know that there exist 
$\theta_{a}$,  $T_{a}\in \R$ such that 
\[e^{-i\theta}u(t, \cdot+y_{0})=W^{a}(t)=e^{ i \theta_{a}}W^{+1}(t-T_{a}),
\]
which implies item (i).

As for item (ii), this follows from the variational characterization of $P_{\omega}$ established in \cite{KillipOhPoVi2017}. \end{proof}

\appendix

\section{Proof of Lemma~\ref{SpecLL}}\label{S:A}

In this section, we provide the proof of Lemma~\ref{SpecLL}.  Noting that $\overline{\L v}= -\L(\overline{v})$,  we find that if $\lambda_{1}>0$  is an eigenvalue of the operator $\L$ with the eigenfunction $e_{+}$, then $-\lambda_{1}$ is also an eigenvalue of $\L$ with eigenfunction $e_{-}=\overline{e_{+}}$.  We put $e_{1}=\RE e_{+}$ and $e_{2}=\IM e_{+}$. To show the existence of $e_{+}$, we must study the system 
\begin{equation}\label{SiE}
\left\{
\begin{aligned} 
L_{+} e_{1}&=\lambda_{1}e_{2},\\
-L_{-} e_{2}&=\lambda_{1}e_{1}.
\end{aligned}
\right.
\end{equation}
Recall that $L_{-}$ is self-adjoint on $L^{2}$ and nonnegative (see  \eqref{Lmenos}). Thus, we see that
$L_{-}$ has a unique square root $(L_{-})^{\frac{1}{2}}$ with domain $H^{1}$. Consider the operator $\TT$ on $L^{2}$ 
with domain $H^{4}$,
\[
\TT=(L_{-})^{\frac{1}{2}}(L_{+})(L_{-})^{\frac{1}{2}}.
\]
As $P_{\omega}$ has exponential decay, it follows that $\TT$ is a relatively compact, self-adjoint, perturbation of $(-\Delta+\omega)^{2}$.
Thus, by the Weyl Theorem, we know that $\sigma_{\text{ess}}(\TT)=[\omega, \infty)$.
Now, suppose that there exists $g\in H^{4}$ such that
\begin{equation}\label{egenT}
\TT g=-\lambda^{2}_{1}g.
\end{equation}
Then, taking
\[
e_{1}:=(L_{-})^{\frac{1}{2}}g, \quad \text{and} \quad e_{2}:=\frac{1}{\lambda_{1}}(L_{+})(L_{-})^{\frac{1}{2}}g
\]
we obtain a solution to \eqref{SiE}, which implies the existence of the eigenfunction $e_{+}$. 

Thus, to show the existence of $e_{+}$ we need to show that $\TT$ has at least one negative eigenvalue $-\lambda^{2}_{1}$.
Indeed, we have the following:

\begin{lemma}\label{NegativeL}
\[
\Lambda(\TT):=\inf\left\{(\TT g, g)_{L^{2}}, g\in H^{4}, \|g\|_{L^{2}}=1 \right\}<0.
\]
\end{lemma}
\begin{proof}
We recall that  $L_{+}$ has only one negative eigenvalue $-\lambda_{1}$ with a corresponding  
eigenfunction $Z\in H^{2}(\R^{3})$ (it is radial and positive function).
 We define
\[
\Phi=Z+\mu\, x\cdot \nabla P_{\omega},
\quad\text{with}\quad \mu=-\frac{(Z, P_{\omega})_{L^{2}}}{(x\cdot \nabla P_{\omega}, P_{\omega})_{L^{2}}}
=\frac{(Z, P_{\omega})_{L^{2}}}{\tfrac{3}{2}\|P_{\omega}\|^{2}_{L^{2}}}.
\]
Notice that $\mu> 0$ and $(\Phi, P_{\omega})_{L^{2}}=0$.

Now we show that
\begin{align*}
	N_{0}:=\int_{\R^{3}}L_{+}(\Phi)\Phi dx<0.
\end{align*}
Indeed, we easily see that
\begin{align}\nonumber
	\<L_{+}(\Phi),\Phi\>&=
	\<L_{+}(Z),Z\>+\mu^{2}\<L_{+}(x\cdot \nabla P_{\omega}),x\cdot \nabla P_{\omega}\>
	-2\mu\lambda_{1}\<Z,x\cdot \nabla P_{\omega}\>\\ \label{NegaL}
	&=\<L_{+}(Z),Z\>-\mu^{2}\|\nabla P_{\omega}\|^{2}_{L^{2}}-2\mu\lambda_{1}\<Z,x\cdot \nabla P_{\omega}\>.
\end{align}
We claim that $\<Z,x\cdot \nabla P_{\omega}\>\geq 0$. Suppose by contradiction that $\<Z,x\cdot \nabla P_{\omega}\><0$. Then we have
\[
\<Z,L_{+}(x\cdot \nabla P_{\omega})\>=
\<L_{+}(Z), x\cdot \nabla P_{\omega}\>=-\lambda_{1}\<Z, x\cdot \nabla P_{\omega}\>>0.
\]
This implies $Z$ satisfies the condition \eqref{QortoX334} in Lemma~\ref{CoerQuadra}. Thus, as $\<Z, \partial_{j}P_{\omega}\>=0$ for $j=1,2,3$, by Lemma~\ref{CoerQuadra}
we deduce that $\<L_{+}(Z),Z\>>0$, which is a contradiction. Having shown that $\<Z,x\cdot \nabla P_{\omega}\>\geq 0$, by \eqref{NegaL} we obtain that $N_{0}<0$.

Next, we have (recall that $(\Phi, P_{\omega})_{L^{2}}=0$)
\[
((L_{-}+1)\Phi, P_{\omega})_{L^{2}}=(\Phi, (L_{-}+1)P_{\omega})_{L^{2}}=0.
\]
Since
\[
\text{Ran}(L_{-})^{\bot}=\text{Ker}(L_{-})=\text{span}\left\{P_{\omega}\right\},
\]
it follows that $(L_{-}+1)\Phi\in \overline{\text{Ran}(L_{-})}$. Then for $\epsilon>0$ (which will be chosen later) there exists $g_{\epsilon}\in H^{2}$ such that
\begin{equation}\label{Gps}
\|L_{-}g_{\epsilon}-(L_{-}+1)\Phi\|_{L^{2}}<\epsilon.
\end{equation}
We set $G_{\epsilon}:=(L_{-}+1)^{-1}L_{-}g_{\epsilon}$. Note that
$G_{\epsilon}\in H^{2}$. By \eqref{Gps}, we infer that
\[
\|G_{\epsilon}-\Phi\|_{H^{2}}\leq \epsilon \|(L_{-}+1)^{-1}\|_{L^{2}\to L^{2}},
\]
which implies that there exists a constant $C_{0}>0$ such that
\[
\left| \int_{\R^{3}}L_{+}(G_{\epsilon})G_{\epsilon}dx-\int_{\R^{3}}L_{+}(\Phi)\Phi dx \right|\leq C_{0}\epsilon.
\]
Choosing $\epsilon=\frac{-N_{0}}{2C_{0}}$, we see that $(L_{+}G_{\epsilon}, G_{\epsilon})<0$. 

Now, if $G=(L_{-})^{-\frac{1}{2}}G_{\epsilon}$, then we have
\[
(\TT G, G)_{L^{2}}=(L_{+}G_{\epsilon}, G_{\epsilon})_{L^{2}}<0.
\]
Since  $\sigma_{\text{ess}}(\TT)=[\omega, +\infty)$, we conclude that the operador $\TT$ has at least one 
negative eigenvalue.
\end{proof}

%
%
%
%

%

Now, using the same argument developed in \cite[Subsection 7.2.2]{DuyMerle2009} we deduce that $e_{\pm}\in \Sch(\R^{3})$.
We observe that discussions given in Section~\ref{S:Spectral} show that 
\begin{equation}\label{KerL}
\text{ker}\left\{\L\right\}
=\text{span}\left\{\partial_{1}P_{\omega}, \partial_{2}P_{\omega}, \partial_{3}P_{\omega}, iP_{\omega}  \right\}.
\end{equation}
Moreover, by Lemma~\ref{NegativeL}  we see that $\left\{\pm \lambda_{1}\right\}\subset \sigma(\L)$.

Finally, we will characterize the real spectrum of $\L$. First, notice that 
\[
\L=JL, \quad 
J:=\begin{pmatrix}
0 & 1 \\
-1 & 0 
\end{pmatrix}
\quad \text{and}\quad 
L:=\begin{pmatrix}
L_{+} & 0 \\
0 & L_{-} 
\end{pmatrix}
\]
As the operator $L$ is a compact perturbation
of
\[
\begin{pmatrix}
-\Delta+\omega &  0\\
0 & -\Delta+\omega 
\end{pmatrix}
\]
we find that $\sigma_{\text{ess}}(\L)=\left\{i\xi: \xi\in \R, |\xi|\geq \omega\right\}$.
Consequently, $\sigma(\L)\cap (\R \setminus\left\{0\right\})$ contains only eigenvalues. It remains to show $\sigma(\L)\cap (\R \setminus\left\{0\right\})=\left\{-\lambda_{1}, \lambda_{1}\right\}$.

Assume towards a contradiction that there instead exists $f\in H^{2}$ such that
\[
\L f=-\lambda_{0}f,
\]
with $\lambda_{0}\in \R\setminus\left\{0, -\lambda_{1}, \lambda_{1}\right\}$. Since $\F(\L g, h)=-\F(g, \L h)$ we see that
\[
(\lambda_{1}+\lambda_{0})\F(f, e_{+})=(\lambda_{1}-\lambda_{0})\F(f, e_{-})=0
\quad \text{and}\quad 
\lambda_{0}\F(f,f)=-\lambda_{0}\F(f,f).
\]
Therefore,
\[
\F(f, e_{+})=\F(f, e_{-})=\F(f,f)=0.
\]
Equivalently, we have
\[
(f, e_{+})_{L^{2}}=(f, e_{-})_{L^{2}}=0.
\]
Next, we write
\[
f=i\beta_{0}P_{\omega}+\sum^{3}_{j=1}\beta_{j}(\partial_{j}P_{\omega})+g,
\quad
\beta_{0}=\tfrac{(f,iP_{\omega})_{L_{2}}}{\|P_{\omega}\|^{2}_{L^{2}}},
\quad
\beta_{j}=\tfrac{(f,\partial_{j}P_{\omega})_{L_{2}}}{\|\partial_{j}P_{\omega}\|^{2}_{L^{2}}},
\]
where $j=1,\ldots,3$ and $g\in Y^{\bot}$. By Remark~\ref{PLf} we get $\F(f,f)=\F(g,g)$. Then,
 Lemma~\ref{NewCoer} implies 
\[
\|g\|^{2}_{H^{1}}\lesssim \F(g,g)=\F(f,f)=0.
\]
In particular, $g=0$ and (see \eqref{KerL})
\[
\lambda_{0}f=\L f =\L\big(i\beta_{0}P_{\omega}+\sum^{3}_{j=1}\beta_{j}(\partial_{j}P_{\omega})\big)=0,
\]
which is a contradiction.

\section{Scattering at the threshold}\label{S:B}

In this section, we demonstrate the existence of scattering threshold solutions.

\begin{proposition}\label{P:exist-scatter}
Assume $E(P_{\omega})>0$. There exists a solution $v$ of \eqref{NLS} such that
$E(v)=E(P_{\omega})$, $M(v)=M(P_{\omega})$, $V(v_{0})>0$ and
$\|v \|_{L^{10}_{t, x}(\R\times\R^{3})}<\infty$. In particular, 
$v$ scatters in both time directions.
\end{proposition}
\begin{proof}
First we claim that there exists $\phi\in C_{c}^{\infty}(\R^{3})$ so that
\begin{equation}\label{EMEq}
E(\phi)=\tfrac{1}{2}E(P_{\omega})
\quad\text{and}
\quad
M(\phi)=\tfrac{1}{2}M(P_{\omega}).
\end{equation}
Indeed, consider $\psi\in C_{c}^{\infty}(\R^{3})$ with $M(\psi)>0$. We define $\psi_{\lambda}(x)=\lambda\psi(\lambda x)$, where $\lambda=\frac{2 M(\psi)}{M(P_{\omega})}$. Notice that $M(\psi_{\lambda})=\frac{1}{2}M(P_{\omega})$.
Now, consider $f_{a}(x)=a^{\frac{3}{2}}f(ax)$, where $f(x)=\psi_{\lambda}(x)$ and $a>0$. Then $M(f_{a})=M(f)=\frac{1}{2}M(P_{\omega})$,
and
\[
E(f_{a})=\tfrac{a^{2}}{2}\|\nabla f\|^{2}_{L^{2}}+\tfrac{a^{6}}{6}\| f\|^{6}_{L^{6}}-
\tfrac{a^{3}}{4}\| f\|^{4}_{L^{4}}.
\]
Since $\lim_{a\to 0^{+}}E(f_{a})=0$ and $\lim_{a\to \infty}E(f_{a})=\infty$, there exists $a_{0}>0$ such that $E(f_{a_{0}})=\frac{1}{2}E(P_{\omega})>0$. Thus, writing $\phi(x)=f_{a_{0}}(x)$
we get \eqref{EMEq}.

By the claim above we see that there exist $\phi^{1}$, $\phi^{2}\in C_{c}^{\infty}(\R^{3})$ satisfying \eqref{EMEq}, with  
\begin{align*}
&E(P_{\omega})=E(\phi^{1})+E(\phi^{2})
 \quad \text{and} \quad
M(P_{\omega})=M(\phi^{1})+M(\phi^{2}).
\end{align*}

Now consider sequences $\{x_n^1\}$ and $\{x_n^2\}$ in $\R^3$ with 
$|x^{1}_{n}-x^{2}_{n}|\to \infty$ as $n\to \infty$.
We set
\[
\phi_{n}^{j}(x)=\phi^{j}(x-x^{j}_{n}), \quad \text{for $j=1$, $2$,}
\]
Writing
\begin{equation}\label{Dpe}
\varphi_{n}(x)=\phi_{n}^{1}(x)+\phi_{n}^{2}(x),
\end{equation}
as $|x^{1}_{n}-x^{2}_{n}|\to \infty$ as $n\to \infty$,  we have 
\begin{align}\label{EMCC}
E(\varphi_{n})=E(P_{\omega}) \quad \text{and} \quad M(\varphi_{n})=M(P_{\omega})
\end{align}
for all $n$ large. Now, for $j=1$, $2$, we define $v^{j}$ to be the global solution of \eqref{NLS}
	with the initial data $v^{j}(0)=\phi^{j}$. Moreover, we define
\[v^{j}_{n}(t,x):=v^{j}(t, x-x^{j}_{n}), \quad j=1,2.\]

By Theorem~\ref{MainTheorem}, we infer that, for $j=1$, $2$,
\[
\|v^{j}_{n} \|_{L^{10}_{t, x}(\R\times\R^{3})}\lesssim_{P_{\omega}} 1 \quad \text{for all $n$.}
\]
Moreover, by persistence of regularity we get (cf. \cite[Lemma 6.2]{KillipOhPoVi2017})
\begin{equation}\label{Estrii}
\| v^{j}_{n}  \|_{L^{10}_{t,x}(\R\times\R^{3})}+\|\nabla v^{j}_{n}\|_{L^{\frac{10}{3}}_{t,x}(\R\times\R^{3})}
\lesssim_{P_{\omega}}1
\end{equation}
and
\begin{equation}\label{masses}
\| v^{j}_{n}\|_{L^{\frac{10}{3}}_{t,x}(\R\times\R^{3})}\lesssim_{P_{\omega}}1
\end{equation}
for all $n$.

Following the same argument developed in \cite[Lemma 9.2]{KillipOhPoVi2017} and using orthogonality condition $|x^{1}_{n}-x^{2}_{n}|\to \infty$ as $n\to \infty$, we have that 
\begin{equation}\label{Apliq}
\lim_{n\rightarrow\infty}[\|  v^{1}_{n}  v^{2}_{n}  \|_{L^{5}_{t,x}}+
\|  v^{1}_{n}  \nabla v^{2}_{n}  \|_{L^{\frac{5}{2}}_{t,x}}+\| \nabla v^{1}_{n}  \nabla v^{2}_{n}  \|_{L^{\frac{5}{3}}_{t,x}}
+\|  v^{1}_{n}   v^{2}_{n}  \|_{L^{\frac{5}{3}}_{t,x}}]=0.
\end{equation}

We now define the approximate solution to \eqref{NLS} by
\[W_{n}(t):=v^{1}_{n}(t)+v^{2}_{n}(t).\]
It is clear that  (cf. \eqref{Dpe})
\begin{equation}\label{Aps}
W_{n}(0)=\varphi_{n}.
\end{equation}
Moreover, we have the following global space bound
\begin{equation}\label{Gstb}
\limsup_{n\rightarrow\infty}\left[\| W_{n}  \|_{L_{t}^{\infty}H^{1}_{x}(\R\times\R^{3})}+ \| W_{n}  \|_{L^{10}_{t,x}(\R\times \R^{3})} 
+\| W_{n}  \|_{L_{t}^{\frac{10}{3}}H^{1,\frac{10}{3}}_{x}(\R\times \R^{3})} \right]\lesssim_{P_{\omega}} 1.
\end{equation}
and
\begin{equation}\label{Eimpo}
\limsup_{n\rightarrow\infty}\|\nabla [i\partial_{t}W_{n}+\Delta W_{n} -F(W_{n})]  \|_{L^{\frac{10}{7}}_{x}(\R\times \R^{3})}=0,
\end{equation}
where $F(z)=|z|^{4}z-|z|^{2}z$. Indeed, with \eqref{Estrii}, \eqref{masses} and \eqref{Apliq} in hands, 
the proof of \eqref{Gstb} and \eqref{Eimpo}  is essentially the same as in 
\cite[Lemma 9.4]{KillipOhPoVi2017} and \cite[Lemma 9.5]{KillipOhPoVi2017}, respectively.

Let $u_{n}$ be the corresponding solution to \eqref{NLS} with initial data $\varphi_{n}$, then the stability result 
\cite[Proposition 6.3]{KillipOhPoVi2017} implies that
$\|u_{n} \|_{L^{10}_{t, x}(\R\times\R^{3})}<\infty$ for $n$ large. Moreover, by \eqref{EMCC} we have 
$E(u_{n})=E(P_{\omega})$ and  $M(u_{n})=M(P_{\omega})$. Note also that $V(u_{n}(0))>0$. Indeed, if $V(u_{n}(0))=0$, then $u_{n}=P_{\omega}$ modulo symmetries, which is a contradiction.
\end{proof}

We close with a remark about case $E(P_\omega)=0$, which shows that the threshold behaviors will not be the same as in the case $E(P_\omega)>0$.

\begin{remark}\label{RB2}
Assume $E(P_{\omega})=0$. If $E(u_{0})=E(P_{\omega})$ and $M(u_{0})=M(P_{\omega})$, then the solution $u$ to \eqref{NLS}
does not scatter in both time directions. Indeed, suppose by contradiction that there exists $\psi_{+}\in H^{1}$ so that
\[
\lim_{t\to \infty}\|u(t)-e^{it}\psi_{+}\|_{H^{1}}=0.
\]
As $e^{it}\psi_{+} \to 0$ in $L^{4}$ as $t \to \infty$, we infer that
\[
\tfrac{1}{2}\|\nabla \psi_{+}\|^{2}_{L^{2}}\leq 
\lim_{t\to \infty}E(u(t))=E(P_{\omega})=0
\quad M(\psi_{+})=\lim_{t\to \infty}M(u(t))=M(P_{\omega})>0,
\]
which is a contradiction. Therefore, $u$ does not scatter in positive time.
A similar argument shows that $u$ does not scatter in negative time.

%
%
%
%

\end{remark}


\end{document}